\documentclass[12pt]{amsart}
\usepackage{amssymb}
\usepackage[dvips]{graphicx}
\usepackage{psfrag}
\usepackage{xy}

\newtheorem{theorem}{Theorem}[section]
\newtheorem{proposition}[theorem]{Proposition} 
\newtheorem{lemma}[theorem]{Lemma}
\newtheorem{corollary}[theorem]{Corollary} 

\theoremstyle{definition}
\newtheorem{definition}[theorem]{Definition}
\newtheorem{remark}[theorem]{Remark} 
\newtheorem{example}[theorem]{Example}

\newcommand{\CC}{{\mathbb C}}
\newcommand{\HH}{{\mathbb H}}
\newcommand{\LL}{{\mathbb L}}

\newcommand{\PP}{{\mathbb P}}
\newcommand{\QQ}{{\mathbb Q}}
\newcommand{\RR}{{\mathbb R}}
\newcommand{\TT}{{\mathbb T}}

\newcommand{\ZZ}{{\mathbb Z}}

\newcommand{\cF}{{\mathcal F}}
\newcommand{\cL}{{\mathcal L}}
\newcommand{\cN}{{\mathcal N}}

\newcommand{\cP}{{\mathcal P}}
\newcommand{\cU}{{\mathcal U}}
\newcommand{\cV}{{\mathcal V}}
\newcommand{\cW}{{\mathcal W}}

\newcommand{\fg}{{\mathfrak g}}

\newcommand{\e}{\varepsilon}

\begin{document}

\vspace*{7ex}

\title{Symplectic Origami}

\author{A. Cannas da Silva, V. Guillemin, and A. R. Pires}

\address{Department of Mathematics, Princeton University,
Princeton, NJ 08544-1000, USA and Departamento de Matem\'atica,
Instituto Superior T\'ecnico, 1049-001 Lisboa, Portugal}
\email{acannas@math.princeton.edu}

\address{Department of Mathematics, Massachussets Institute of
Technology, 77 Massachussets Avenue, Cambridge, MA 02139-4307, USA}
\email{vwg@math.mit.edu}

\address{Department of Mathematics, Massachussets Institute of
Technology, 77 Massachussets Avenue, Cambridge, MA 02139-4307, USA}
\email{arita@math.mit.edu}

\date{Revised November 15, 2010.
To appear in {\em Int.\ Math.\ Res.\ Not.}\\
First published online December 2, 2010.
DOI: 10.1093/imrn/rnq241}

\maketitle

\begin{abstract}
An origami manifold is a manifold equipped with a closed 2-form
which is symplectic except on a hypersurface where it is like
the pullback of a symplectic form by a folding map and
its kernel fibrates with oriented circle fibers over a compact base.
We can move back and forth between origami
and symplectic manifolds using cutting (unfolding)
and radial blow-up (folding), modulo compatibility conditions.
We prove an origami convexity theorem for hamiltonian torus
actions, classify toric origami manifolds by polyhedral objects
resembling paper origami and discuss examples.
We also prove a cobordism result and compute the cohomology
of a special class of origami manifolds.
\end{abstract}

\vspace*{3ex}


\section{Introduction}

This is the third in a series of papers on \emph{folded} symplectic manifolds.
The first of these papers~\cite{ca-gu-wo:unfolding} contains a description
of the basic local and semi-global features of these manifolds and
of the \emph{folding} and \emph{unfolding} operations;
in the second~\cite{ca:ffff} it is shown that a manifold is folded
symplectic if and only if it is stable complex
and, in particular, that every oriented 4-manifold is folded
symplectic. (Other recent papers on the topology of
folded symplectic manifolds are~\cite{baykur} and~\cite{bergmann}.)

In this third paper we take up the theme of hamiltonian group actions
on folded symplectic manifolds.
We focus on a special class of folded symplectic manifolds which we call
\emph{origami manifolds}. (Jean-Claude Hausmann pointed out to us
that the term ``origami'' had once been proposed for another
class of spaces: what are now known as orbifolds.)
For the purposes of this introduction, let us say that a
folded symplectic manifold is a triple $(M,Z,\omega)$
where $M$ is an oriented $2n$-dimensional manifold,
$\omega$ a closed 2-form and $Z\stackrel{i}{\hookrightarrow} M$ a hypersurface.
``Folded symplectic'' requires that $\omega$ be symplectic
on $M \smallsetminus Z$ and that the restriction of $\omega$ to $Z$
be odd-symplectic, i.e.
\[
   (i^*\omega)^{n-1}\neq0 \ .
\]
From this one gets on $Z$ a \emph{null foliation}
by lines and $(M,Z,\omega)$ is ``origami''
if this foliation is fibrating with compact connected oriented fibers.
In this case one can \emph{unfold} $M$ by taking the
closures of the connected components of $M\smallsetminus Z$
and identifying boundary points on the same leaf of the null foliation.
We will prove that this unfolding defines a cobordism between
(a compact) $M$ and a disjoint union of (compact) symplectic manifolds $M_i$:
\begin{equation}\label{cob}
   M\sim\displaystyle{\bigsqcup_i M_i} \ .
\end{equation}
Moreover, if $M$ is a hamiltonian $G$-manifold we will prove
that the $M_i$'s are as well.
The \emph{origami} results of this paper involve reconstructing
the moment data of $M$ (and in the toric case $M$ itself)
from the moment data of the $M_i$'s.

Precise definitions of ``folded symplectic'' and ``origami''
are given in Section~\ref{ss:origami}.
In~\ref{ss:cutting} we describe in detail the unfolding
operation (\ref{cob})
and in~\ref{ss:blowup} how one can refold the terms
on the right to reconstruct $M$ via a radial blow-up operation.
Then in Sections~\ref{ss:cutting_blowup} and~\ref{ss:blowup_cut}
we prove that folding and unfolding are inverse operations:
unfolding followed by folding gives one the manifold
one started with and vice-versa.

We turn in Section~\ref{sec:polytopes} to the main theme of this paper:
torus actions on origami manifolds.
In~\ref{sec:convexity} we define for such actions an origami version
of the notion of moment polytope, which turns out to be a
collection of convex polytopes with compatibility conditions,
or {\em folding instructions} on facets.
We then concentrate in Section~\ref{ss:toric} on the toric case
and prove in~\ref{ss:classification}
an origami version of the Delzant theorem.
More explicitly, we show that toric origami manifolds are classified
by \emph{origami templates}: pairs $(\cP,\cF)$,
where $\cP$ is a finite collection of oriented $n$-dimensional Delzant polytopes
and $\cF$ a collection of pairs of facets of these polytopes
satisfying:
\begin{itemize}
\item[(a)]
for each pair of facets $\left\{F_1,F_2\right\} \in \cF$ the corresponding
polytopes in $\cP$ have opposite orientations and are identical in a 
neighborhood of these facets;
\item[(b)]
if a facet occurs in a pair,
then neither itself nor any of its neighboring facets occur in any other pair;
\item[(c)]
the topological space constructed from the disjoint union
of all the $\Delta_i\in\cP$ by identifying facet pairs in $\cF$
is connected.
\end{itemize}

Without the assumption that $M$ be origami, i.e.,
that the null foliation be fibrating,
it is \emph{not} possible to classify
hamiltonian torus actions on folded symplectic manifolds by a finite set
of combinatorial data;
why not is illustrated by example~\ref{exotic}.
Nonetheless, Chris Lee has shown
that a (more intricate) classification of these objects
by moment data \emph{is} possible at least in dimension four~\cite{lee:thesis}.
This result of his we found very helpful in putting our own results
into perspective.

Throughout this introduction we have been assuming that our
ori\-gami manifolds are oriented. However, all the definitions and results 
extend to the case of nonorientable origami manifolds, and that is how 
they will be presented in this paper. In particular, the notion of origami 
template explained above becomes that of Definition~\ref{def:template}, 
which drops the orientations of the polytopes in $\cP$ and allows for sets 
of single facets in $\cF$.
Moreover, as we show in Section~\ref{sec:polytopes}, some of the most 
curious examples of origami manifolds (such as $\RR\PP^{2n}$ and the 
Klein bottle) are nonorientable.

The final two sections of this paper contain results that hold only 
for oriented origami manifolds.

In Section~\ref{cobordism} we prove 
that (\ref{cob}) is a cobordism
and, in fact, an equivariant cobordism in presence of group actions.
We show that this cobordism is a \emph{symplectic cobordism},
i.e., there exists a closed two form on the cobording manifold
whose restriction to $M$ is the folded symplectic form on $M$
and on the symplectic cut pieces is the symplectic form on those manifolds.
Moreover, in the presence of a (hamiltonian) compact group action,
this cobordism is a (hamiltonian) equivariant cobordism.
Using these results and keeping track of stable almost complex structures,
one can give in the spirit of~\cite{gu-gi-ka:cobordisms}
a proof that the equivariant spin-$\CC$ quantization of $M$
is, as a virtual vector space (and in the presence of group actions
as a virtual representation), equal to the spin-$\CC$ quantizations
of its symplectic cut pieces.
However, we will not do so here.
We refer the reader instead to the proof of this result in
Section~8 of~\cite{ca-gu-wo:unfolding},
which is essentially a cobordism proof of this type.

Section~\ref{sec:cohomology} is devoted to the origami version
of a theorem in the standard theory of hamiltonian actions:
In it we compute the cohomology groups of an oriented toric origami manifold,
under the assumption that the folding hypersurface be connected.

Origami manifolds and higher codimension analogues arise naturally
when converting hamiltonian torus actions on symplectic manifolds into
free actions by generalizations of \textit{radial blow-up}
along orbit-type strata.
We intend to pursue this direction to obtain free hamiltonian torus
actions on compact presymplectic manifolds, complementing
recent work by Yael Karshon and Eugene Lerman on non-compact
symplectic toric manifolds~\cite{ka-le:noncompact}.


\section{Origami Manifolds}
\label{sec:origami}

\subsection{Folded symplectic and origami forms}
\label{ss:origami}

\begin{definition}
A \textit{folded symplectic form} on a $2n$-dimensional manifold $M$
is a closed 2-form $\omega$ whose top power $\omega^n$ vanishes
transversally on a submanifold $Z$, called the
\textit{folding hypersurface} or \textit{fold},
and whose restriction to that submanifold has maximal rank.
The pair $(M,\omega)$ is then called a
\textit{folded symplectic manifold}.
\end{definition}

By transversality, the folding hypersurface $Z$ of a folded symplectic manifold
is indeed of codimension one and embedded.
An analogue of Darboux's theorem for folded symplectic
forms~\cite{ca-gu-wo:unfolding,ma:formes}
says that near any point $p \in Z$ there is a coordinate chart
centered at $p$ where the form $\omega$ is
\[
   x_1 dx_1 \wedge dy_1 + dx_2 \wedge dy_2 + \ldots + dx_n \wedge dy_n \ .
\]

Let $(M,\omega)$ be a $2n$-dimensional folded symplectic manifold.
Let $i : Z \hookrightarrow M$ be the inclusion of the
folding hypersurface $Z$.
Away from $Z$, the form $\omega$ is nondegenerate,
so $\left. \omega^n \right|_{M \smallsetminus Z} \neq 0$.
The induced restriction $i ^* \omega$ has a one-dimensional kernel at
each point: the line field $V$ on $Z$, called the
\textit{null foliation}.
Note that $V = TZ \cap E \subset i^* TM$ where $E$ is the rank 2
bundle over $Z$ whose fiber at each point is the kernel of $\omega$.

\begin{remark}
When a folded symplectic manifold $(M,\omega)$ is an oriented
manifold, the complement $M \smallsetminus Z$ decomposes into
open subsets $M^+$ where $\omega^n > 0$ and $M^-$ where $\omega^n < 0$.
This induces a coorientation on $Z$ and hence an orientation on $Z$.
From the form $\left( i^* \omega \right)^{n-1}$ we obtain
an orientation of the quotient bundle $\left( i^* TM \right)/E$
and hence an orientation of $E$.
From the orientations of $TZ$ and of $E$, we obtain an orientation
of their intersection, the null foliation $V$.
\end{remark}

We concentrate on the case of fibrating null foliation.

\begin{definition}
An \textit{origami manifold} is a folded
symplectic manifold $(M, \omega)$ whose null foliation
is fibrating with oriented circle fibers, $\pi$,
over a compact \textit{base}, $B$. (It would be natural
to extend this definition admitting {\em Seifert fibrations}
and {\em orbifold bases}.)
\[
\begin{array}{l}
   Z \\
   \downarrow \pi \\
   B
\end{array}
\]
The form $\omega$ is called an \textit{origami form}
and the null foliation, i.e., the vertical bundle of $\pi$,
is called the \textit{null fibration}.
\end{definition}

\begin{remark}
When an origami manifold is oriented we assume that
any chosen orientation of the null fibration or any principal
$S^1$-action matches the induced orientation of the null foliation $V$.
By definition, a nonorientable origami manifold still has
an orientable null foliation.
\end{remark}

Notice that, on an origami manifold,
the base $B$ is naturally symplectic:
As in symplectic reduction, there is a unique
symplectic form $\omega_B$ on $B$ satisfying
\[
   i^* \omega = \pi^* \omega_B \ .
\]

Notice also that the fold $Z$ is necessarily compact
since it is the total space of a circle fibration with compact base.
We can choose different principal $S^1$-actions on $Z$ by choosing
nonvanishing (positive) vertical vector fields with periods $2 \pi$. 

\begin{example}
\label{ex:spheres}
Consider the unit sphere $S^{2n}$
in euclidean space $\RR^{2n+1} \simeq \CC^{n} \times \RR$
with coordinates $x_1,y_1,\ldots,x_n,y_n,h$.
Let $\omega_0$ be the restriction to $S^{2n}$ of
$dx_1 \wedge dy_1 + \ldots + dx_n \wedge dy_n
= r_1 dr_1 \wedge d\theta_1 + \ldots + r_n dr_n \wedge d\theta_n$.
Then $\omega_0$ is a folded symplectic form.
The folding hypersurface is the equator sphere given by the
intersection with the plane $h=0$.
The null foliation is the Hopf foliation since
\[
   \imath_{_{\frac{\partial}{\partial \theta_1}+ \ldots
    +\frac{\partial}{\partial \theta_n}}}
   \omega_0 = - r_1 dr_1 - \ldots - r_n dr_n
\]
vanishes on $Z$, hence a null fibration is
$S^1 \hookrightarrow S^{2n-1} \twoheadrightarrow \CC \PP^{n-1}$.
Thus, $(S^{2n},\omega_0)$ is an orientable origami manifold.
\end{example}

\begin{example}
\label{ex:nonorientable}
The standard folded symplectic form $\omega_0$
on $\RR \PP^{2n} = S^{2n}/\ZZ_2$
is induced by the restriction to $S^{2n}$ of the $\ZZ_2$-invariant
form $dx_1 \wedge dx_2 + \ldots + dx_{2n-1} \wedge dx_{2n}$
in $\RR^{2n+1}$~\cite{ca-gu-wo:unfolding}.
The folding hypersurface is
$\RR \PP^{2n-1} \simeq \{ [x_1,\ldots,x_{2n},0] \}$,
a null fibration is the $\ZZ_2$-quotient of the Hopf fibration
$S^1 \hookrightarrow \RR \PP^{2n-1} \twoheadrightarrow \CC \PP^{n-1}$,
and $(\RR \PP^{2n},\omega_0)$ is a nonorientable origami manifold.
\end{example}

The following definition regards \textit{symplectomorphism}
in the sense of \textit{presymplectomorphism}.

\begin{definition}
Two (oriented) origami manifolds
$(M,\omega)$ and $(M',\omega')$
are \textit{symplectomorphic} if there is a (orientation-preserving)
diffeomorphism $\rho : M \to \widetilde{M}$
such that $\rho^* \widetilde{\omega} = \omega$.
\end{definition}

This notion of equivalence between origami manifolds
stresses the importance of the null foliation being fibrating,
and not a particular choice of principal circle fibration.
We might sometimes identify symplectomorphic origami manifolds.

\subsection{Cutting}
\label{ss:cutting}

The folding hypersurface $Z$ plays the role of an
{\em exceptional divisor} as it can be {\em blown-down}
to obtain honest symplectic pieces.
(Origami manifolds may hence
be interpreted as {\em birationally symplectic manifolds}.
However, in algebraic geometry
the designation {\em birational symplectic manifolds} was
used by Huybrechts~\cite{hu:birational} in a different context,
that of birational equivalence for complex manifolds
equipped with a holomorphic nondegenerate two-form.)
This process, called \textit{cutting}
(or {\em blowing-down} or {\em unfolding}),
is essentially symplectic cutting and
was described in~\cite[Theorem 7]{ca-gu-wo:unfolding}
in the orientable case.

\begin{example}
\label{ex:cutting_spheres}
Cutting the origami manifold
$(S^{2n}, \omega_0)$ from Example~\ref{ex:spheres}
produces $\CC \PP^n$ and $\overline{\CC \PP^n}$
each equipped with the same multiple of the Fubini-Study form
with total volume equal to that of an original hemisphere, $n! (2\pi)^n$.
\end{example}

\begin{example}
\label{ex:cutting_nonorientable}
Cutting the origami manifold
$(\RR \PP^{2n}, \omega_0)$ from Example~\ref{ex:nonorientable}
produces a single copy of $\CC \PP^n$.
\end{example}

\begin{proposition}
\label{pro:cutting}
\cite{ca-gu-wo:unfolding}
Let $(M^{2n},\omega)$ be an oriented origami manifold.

Then the unions
$M^+ \sqcup B$ and $M^- \sqcup B$ each admits a structure of
$2n$-dimensional symplectic manifold, denoted $(M^+_0,\omega^+_0)$
and $(M^-_0,\omega^-_0)$ respectively, with $\omega^+_0$ and
$\omega^-_0$ restricting to $\omega$ on $M^+$ and $M^-$
and with a natural embedding of $(B, \omega_B)$ as a symplectic
submanifold with radially projectivized normal bundle isomorphic to
the null fibration $Z \stackrel{\pi}{\to} B$.

The orientation induced from the original orientation on $M$
matches the symplectic orientation on $M^+_0$ and is opposite
to the symplectic orientation on $M^-_0$.
\end{proposition}

The proof relies on origami versions of Moser's trick
and of Lerman's cutting.
Lerman's cutting applies to a hamiltonian circle action,
defined as in the symplectic case:

\begin{definition}
The action of a Lie group $G$
on an origami manifold $(M,\omega)$ is \textit{hamiltonian} if it admits
a \textit{moment map}, $\mu : M \to \fg^*$, satisfying the conditions:
\begin{itemize}
\item
$\mu$ collects hamiltonian functions, i.e.,
$d \langle \mu , X \rangle = \imath_{X^{\#}} \omega$,
$\forall X \in \fg := \text{Lie} (G)$, where $X^{\#}$ is the
vector field generated by $X$;
\item
$\mu$ is equivariant with respect to the given action
of $G$ on $M$ and the coadjoint action of $G$ on
the dual vector space $\fg^*$.
\end{itemize}
$(M,\omega,G,\mu)$ denotes an origami manifold equipped with
a hamiltonian action of a Lie group $G$ having moment map $\mu$.
\end{definition}

Moser's trick needs to be adapted as in~\cite[Theorem 1]{ca-gu-wo:unfolding}. 
We start from a tubular neighborhood model defined as follows.

\begin{definition}\label{def:mosermodel}
A \textit{Moser model} for an oriented origami manifold $(M,\omega)$
with null fibration $Z \stackrel{\pi}\rightarrow B$ is a diffeomorphism
\[
   \varphi:Z\times(-\varepsilon,\varepsilon)\to\cU
\]
where $\varepsilon>0$ and $\cU$ is a tubular neighborhood of $Z$ 
such that $\varphi(x,0)=x$ for all $x\in Z$ and
\[
   \varphi^* \omega = p^* i^* \omega
   + d \left( t^2 p^* \alpha \right) \ ,
\]
with $p: Z \times (-\varepsilon,\varepsilon) \to Z$
the projection onto the first factor,
$i : Z \hookrightarrow M$ the inclusion,
$t$ the real coordinate on the interval $(-\varepsilon,\varepsilon)$
and $\alpha$ an $S^1$-connection form for a chosen
principal $S^1$-action along the null fibration.
\end{definition}

A choice of a principal $S^1$-action along the null fibration,
$S^1 \hookrightarrow Z \stackrel{\pi}{\to} B$, corresponds to
a vector field $v$ on $Z$ generating the principal $S^1$-bundle.
Following~\cite[Theorem 1]{ca-gu-wo:unfolding}, a Moser model can then
be found after choices of a connection form $\alpha$,
a small enough positive real number $\varepsilon$ and
a vector  field $w$ over a tubular neighborhood of $Z$ such that,
at each $x\in Z$,
the pair $(w_x,v_x)$ is an oriented basis of the kernel of $\omega_x$.
The orientation on this kernel is determined by the given orientation of $TM$
and the symplectic orientation of $TM$ modulo the kernel.
Conversely, a Moser model for an oriented origami manifold
gives a connection 1-form $\alpha$ by contracting 
$\varphi^* \omega$ with the vector field
$\frac{1}{2t} \frac{\partial}{\partial t}$
(and hence gives a vertical vector field $v$
such that $\imath_v \alpha = 1$ which generates an $S^1$-action),
an $\varepsilon$ from the width of the symmetric real interval
and a vector field $w=\varphi_*\left(\frac{\partial}{\partial t}\right)$.

\begin{lemma}
Any two Moser models
$\varphi_i:Z\times(-\varepsilon_i,\varepsilon_i)\to\cU_i$, with $i=0,1$, 
admit isotopic restrictions to $Z\times(-\varepsilon,\varepsilon)$
for $\varepsilon$ small enough, that is,
those restrictions can be smoothly connected by a family of Moser models.
\end{lemma}

\begin{proof}
Let $v_0$ and $v_1$ be the vector fields generating the $S^1$-actions
for models $\varphi_0$ and $\varphi_1$, and let $w_0$ and $w_1$
be the vector fields 
$w_i = (\varphi_i)_*\left(\frac{\partial}{\partial t}\right)$.
The vector fields $v_t = (1-t)v_0 + t v_1$ on $\cU_0\cap\cU_1$
correspond to a connecting family of $S^1$-actions all with the same orbits,
orientation and periods $2 \pi$.
Connect the vector fields $w_0$ and $w_1$ on
$\cU_0\cap\cU_1$ by a smooth family of vector fields $w_t$
forming oriented bases $(w_t,v_t)$ of $\ker\omega$ over $Z$.
Note that the $v_t$ are all positively proportional and
the set of all possible vector fields $w_t$ is contractible.
By compactness of $Z$, we can even take the convex combination
$w_t=(1-t)w_0+t w_1$, for $\varepsilon$ small enough.
Pick a smooth family of connections $\alpha_t$:
For instance, using a metric pick 1-forms $\beta_t$ such that
$\beta_t (v_t) =1$ and then average each $\beta_t$ by the $S^1$-action
generated by $v_t$.
For the claimed isotopy, use a corresponding family of Moser models
$\varphi_t:Z\times(-\varepsilon,\varepsilon)\to\cU_t$
with $\varepsilon$ sufficiently small so that integral curves of all $w_t$
starting at points of $Z$ are defined for $t\in(-\varepsilon,\varepsilon)$.
\end{proof}

With these preliminaries out of the way, we recall and expand
the proof from~\cite{ca-gu-wo:unfolding} for Proposition~\ref{pro:cutting}.

\begin{proof}
Choose a principal $S^1$-action along the null fibration,
$S^1 \hookrightarrow Z \stackrel{\pi}{\to} B$.
Let $\varphi : Z \times (-\varepsilon,\varepsilon) \to \cU$ be a Moser model,
and let $\cU^+$ denote $M^+ \cap \cU = \varphi (Z \times (0,\varepsilon))$.
The diffeomorphism
\[
   \psi : Z \times (0,\varepsilon^2) \to \cU^+ \ , \qquad
   \psi (x,s) = \varphi(x,\sqrt{s})
\]
induces a symplectic form
\[
    \psi^* \omega = p^* i^* \omega
   + d \left( s p^* \alpha \right) =: \nu
\]
on $Z \times (0,\varepsilon^2)$
that extends by the same formula to $Z \times (-\varepsilon^2,\varepsilon^2)$.

As in standard symplectic cutting~\cite{le:cuts}, form the product
$(Z \times (-\varepsilon^2,\varepsilon^2) , \nu) \times (\CC, -\omega_0)$
where $\omega_0 = \frac{i}{2} dz \wedge d \overline z$.
The product action of $S^1$ on
$Z \times (-\varepsilon^2,\varepsilon^2) \times \CC$ by
\[
   e^{i\theta} \cdot (x,s,z) = (e^{i\theta} \cdot x, s, e^{-i\theta} z)
\]
is hamiltonian and $\mu (x,s,z) = s - \frac{|z|^2}{2}$
is a moment map.
Zero is a regular value of $\mu$ and the corresponding level
is a codimension-one submanifold which decomposes as
\[
   \mu^{-1} (0) = Z \times \{ 0 \} \times \{ 0 \}
   \bigsqcup \{ (x,s,z) \mid s > 0 \ ,
   |z|^2 = 2 s \} \ .
\]
Since $S^1$ acts freely on
$\mu^{-1} (0)$, the quotient $\mu^{-1} (0) / S^1$ is a manifold and the
point-orbit map is a principal $S^1$-bundle.
Moreover, we can view it as
\[
   \mu^{-1} (0) / S^1 \simeq B \sqcup \cU^+ \ .
\]
Indeed, $B$ embeds as a codimension-two submanifold via
\[
   \begin{array}{rcrcl}
   j & : & B & \longrightarrow & \mu^{-1} (0) / S^1 \\
   & & \pi (x) & \longmapsto & [x,0,0] \, \qquad \mbox{ for } x \in Z
   \end{array}
\]
and $\cU^+$ embeds as an open dense submanifold via
\[
   \begin{array}{rcrcl}
   j^+ & : & \cU^+ & \longrightarrow & \mu^{-1} (0) / S^1 \\
   & & \psi (x,s) & \longmapsto & [x,s,\sqrt{2s}] \ .
   \end{array}
\]

The symplectic form $\Omega_{red}$ on $\mu^{-1} (0) / S^1$
obtained by symplectic reduction is such that the above embeddings
of $(B, \omega_B)$ and of $(\cU^+, \omega|_{\cU^+})$
are symplectic.

The normal bundle to $j(B)$ in  $\mu^{-1} (0) / S^1$
is the quotient over $S^1$-orbits (upstairs and downstairs) of the
normal bundle to $Z \times \{ 0 \} \times \{ 0 \}$ in $\mu^{-1} (0)$.
This latter bundle  is the product bundle
$Z \times \{ 0 \} \times \{ 0 \} \times \CC$ 
where the $S^1$-action is
\[
   e^{i\theta} \cdot (x,0,0,z) = (e^{i\theta} \cdot x, 0, 0, e^{-i\theta} z) \ .
\]
Performing $\RR^+$-projectivization and taking the $S^1$-quotient we
get the bundle $Z \to B$ with the isomorphism
\[
   \begin{array}{rcl}
   \big( Z \times \{ 0 \} \times \{ 0 \} \times \CC^* \big) / S^1
   \ni [x,0,0,r e^{i\theta}]
   & \quad \longmapsto \quad & e^{i\theta} x \in Z \\
   \downarrow \phantom{99} & & \phantom{9} \downarrow \\
   \big( Z \times \{ 0 \} \times \{ 0 \} \big) / S^1
   \ni [x,0,0]
   & \quad \longmapsto \quad & \pi (x) \in B \ .
   \end{array}
\]

By gluing the rest of $M^+$ along $\cU^+$, we produce
a $2n$-dimensional symplectic manifold $(M^+_0,\omega^+_0)$
with a symplectomorphism
$\overline{j^+} : M^+ \to M^+_0 \smallsetminus j(B)$ extending $j^+$.

For the other side, the map
$\psi_- : Z \times (0,\varepsilon^2) \to \cU^- := M^-\cap\cU$,
$(x,s) \mapsto \varphi(x,-\sqrt{s})$ reverses orientation and
$(\psi_-)^* \omega = \nu$.
The base $B$ embeds as a
symplectic submanifold of $\mu^{-1} (0) / S^1$ by the previous formula.
The embedding
\[
   \begin{array}{rcrcl}
   j^- & : & \cU^- & \longrightarrow &
   \mu^{-1} (0) / S^1 \\
   & & \psi_- (x,s) & \longmapsto & [x,s,- \sqrt{2s}]
   \end{array}
\]
is an orientation-reversing symplectomorphism.
By gluing the rest of $M^-$ along $\cU^-$, we produce
$(M^-_0,\omega^-_0)$ with a symplectomorphism
$\overline{j^-} : M^- \to M^-_0 \smallsetminus j(B)$ extending $j^-$.
\end{proof}

\begin{remark}
\label{rmk:involution}
The cutting construction in the previous proof produces a symplectomorphism
$\gamma$ between tubular neighborhoods
$\mu^{-1} (0) / S^1$ of the embeddings of $B$ in $M_0^+$ and $M_0^-$,
comprising $\cU^+ \to \cU^-$, $\varphi (x,t) \mapsto \varphi (x,-t)$,
and the identity map on $B$:
\[
\begin{array}{rccl}
   \gamma : & \mu^{-1} (0) / S^1 & \longrightarrow & \mu^{-1} (0) / S^1 \\
   & [x,s,\sqrt{2s}] & \longmapsto & [x,s,-\sqrt{2s}] \ .
\end{array}
\]
\end{remark}

\begin{definition}
Symplectic manifolds $(M^+_0,\omega^+_0)$ and $(M^-_0,\omega^-_0)$
obtained by cutting are called
\textit{symplectic cut pieces} of the oriented origami manifold $(M,\omega)$
and the embedded copies of $B$ are called \textit{centers}.
\end{definition}

The next proposition states that symplectic cut pieces
of an origami manifold are unique up to symplectomorphism.

\begin{proposition}
Different choices of a Moser model for a tubular
neighborhood of the fold in an origami manifold
yield symplectomorphic symplectic cut pieces.
\end{proposition}

\begin{proof}
Let $\varphi_0$ and $\varphi_1$ be two Moser models
for a tubular neighborhood $\cU$ of the fold $Z$ in
an origami manifold $(M,\omega)$.
Let $(M_0^\pm,\omega_0^\pm)$ and $(M_1^\pm,\omega_1^\pm)$
be the corresponding symplectic manifolds obtained by the above cutting.
Let $\varphi_t : Z \times (-\varepsilon,\varepsilon) \to \cU_t$
be an isotopy between (restrictions of) $\varphi_0$ and $\varphi_1$. 
By suitably rescaling $t$, we may assume that $\varphi_t$ is a technical
isotopy in the sense of~\cite[p.89]{br-ja:book},
i.e., $\varphi_t = \varphi_0$ for $t$ near $0$
and $\varphi_t = \varphi_1$ for $t$ near $1$.
Let
\[
   j_t^\pm : \cU_t^\pm \longrightarrow \mu^{-1}(0)/S^1
\]
be the corresponding isotopies of symplectic embeddings,
where $\mu^{-1}(0)/S^1$ is equipped with $\left( \Omega_{red} \right)_t$,
and let $(M_t^\pm,\omega_t^\pm)$ be the corresponding families
of symplectic manifolds obtained from gluing:
For instance, $(M_t^+,\omega_t^+)$ is the quotient of the disjoint union
\[
   M^+ \sqcup \mu^{-1}(0)/S^1
\]
by the equivalence relation which sets each point in $\cU^+$
equivalent to its image by the symplectomorphism $j_t^+$.

Let $\cU_c := \overline{\cU_{out}} \setminus \cU_{in}$
be a compact subset of each $\cU_t$
where $\cU_{in}$ and $\cU_{out}$ are tubular neighborhoods
of $Z$ with $\overline{\cU_{in}} \subset \cU_{out}$.
Let $\cU_*^+ = \cU_* \cap M^+$ where $*$ stands for the subscripts
$c$, $out$ or $in$.
Let $C$ be a compact neighborhood of
$j^+ \left( [0,1] \times \cU_c^+ \right) :=
\cup_{t \in [0,1]} j^+_t \left( \cU_c^+ \right)$ in $\mu^{-1}(0)/S^1$,
such that $B \cap C = \emptyset$.

By Theorem~10.9 in~\cite{br-ja:book}, there is a smooth
family of diffeomorphisms
\[
   H_t : \mu^{-1}(0)/S^1 \longrightarrow \mu^{-1}(0)/S^1 \ ,
   \qquad t \in [0,1] \ ,
\]
which hold fixed all points outside $C$
(in particular, the $H_t$ fix a neighborhood of $B$),
with $H_0$ the identity map and such that
\[
   j_t^+ |_{\cU_c^+} = H_t \circ j_0^+ |_{\cU_c^+} \ .
\]
The diffeomorphism $H_t$ restricted to $D := \overline{j_0^+ (\cU_{out}^+)}$
and the identity diffeomorphism on $M^+ \setminus \cU_{in}^+$
together define a diffeomorphism
\[
   \phi_t : M_0^+ \longrightarrow M_t^+ \ .
\]
All forms in the family $\phi_t^* \omega_t^+$ on $M_0^+$
are symplectic, have the same restriction to $B$ and are
equal to $\omega_0^+$ away from the set $D$ which retracts to $B$.
Hence, all $\phi_t^* \omega_t^+$ are in the same cohomology class and, moreover,
\[
   \frac{d}{dt} \phi_t^* \omega_t^+ = d \beta_t
\]
for some smooth family of 1-forms $\beta_t$ supported
in the compact set $D$~\cite[p.95]{mc-sa}.

By solving Moser's equation
\[
   \imath_{w_t} \omega_t^+ + \beta_t = 0
\]
we find a time-dependent vector field $w_t$ compactly supported on $D$.
The isotopy $\rho_t : M_0^+ \to M_0^+$, $t \in \RR$,
corresponding to this vector field satisfies $\rho_t \equiv \mbox{id}$
away from $D$ and
\[
   \rho_t^* \left( \phi_t^* \omega_t^+ \right) = \omega_0^+
   \qquad \mbox{ for all } t \ .
\]

The map $\phi_1 \circ \rho_1$ is a symplectomorphism
between $(M_0^+,\omega_0^+)$ and $(M_1^+,\omega_1^+)$.
Similarly for $(M_0^-,\omega_0^-)$ and $(M_1^-,\omega_1^-)$.
\end{proof}

Cutting may be performed for any {\em nonorientable} origami
manifold $(M,\omega)$ by
working with its orientable double cover.
The double cover involution yields a symplectomorphism
from one symplectic cut piece to the other.
Hence, we regard these pieces as a trivial double cover
(of one of them) and call their $\ZZ_2$-quotient
the {\em symplectic cut space} of $(M,\omega)$.
In the case where $M \smallsetminus Z$ is connected,
the symplectic cut space is also connected;
see Example~\ref{ex:cutting_nonorientable}.

\begin{definition}
\label{defn:cut_space}
The \textit{symplectic cut space} of an origami manifold $(M,\omega)$
is the natural $\ZZ_2$-quotient of symplectic cut pieces
of its orientable double cover.
\end{definition}

Notice that, when the original origami manifold is compact,
the symplectic cut space is also compact.

\subsection{Radial blow-up}
\label{ss:blowup}

We can reverse the cutting procedure using an
origami (and simpler) analogue of Gompf's gluing
construction~\cite{go:new}.
{\em Radial blow-up} is a local operation
on a symplectic tubular neighborhood of a codimension-two
symplectic submanifold modeled by the following example.

\begin{example}
Consider the standard symplectic $(\RR^{2n},\omega_0)$
with its standard euclidean metric.
Let $B$ be the symplectic submanifold defined by
$x_1=y_1=0$ with unit normal bundle $N$ identified with
the hypersurface $x_1^2 + y_1^2 = 1$.
The map $\beta : N \times \RR \to \RR^{2n}$ defined by
\[
   \beta ((p,e^{i \theta}),r) = p + (r\cos \theta, r\sin \theta, 0, \ldots, 0)
   \mbox{ for } p \in B
\]
induces by pullback an origami form on the cylinder
$N \times \RR \simeq S^1 \times \RR^{2n-1}$, namely
\[
   \beta^* \omega_0 = r dr \wedge d \theta + dx_2 \wedge dy_2
   + \ldots + dx_n \wedge dy_n \ .
\]
\end{example}

Let $(M,\omega)$ be a symplectic manifold with a
codimension-two symplectic submanifold $B$.
Let $i : B \hookrightarrow M$ be the inclusion map.
Consider the radially projectivized normal bundle over $B$
\[
   \cN := \PP^+ \left( i^*TM / TB \right)
   = \{ x \in (i^* TM)/TB , x \neq 0 \} / \sim
\]
where $\lambda x \sim x$ for $\lambda \in \RR^+$.
We choose an $S^1$ action making $\cN$ a principal circle bundle over $B$.
Let $\varepsilon > 0$.

\begin{definition}
A \textit{blow-up model} for a tubular neighborhood $\cU$ of $B$ in
$(M,\omega)$ is a map
\[
   \beta : \cN \times (-\varepsilon, \varepsilon) \longrightarrow \cU
\]
which factors as
\[
   \begin{array}{rrcccl}
   \beta : & \cN \times (-\varepsilon, \varepsilon) &
   \stackrel{\beta_0}{\longrightarrow} &
   \cN \times_{S^1} \CC &
   \stackrel{\eta}{\longrightarrow} & \cU \\
   & (x,t) & \longmapsto & [x,t]
   \end{array}
\]
where $e^{i \theta} \cdot (x,t) = (e^{i \theta} \cdot x , te^{-i \theta})$
for $(x,t) \in \cN \times \CC$ and
$\eta : \beta_0 \left( \cN \times (-\varepsilon, \varepsilon) \right) \to \cU$
is a tubular bundle diffeomorphism.
By \textit{tubular bundle diffeomorphism} we mean a bundle diffeomorphism
covering the identity $B \to B$ and isotopic to a diffeomorphism given
by geodesic flow for some choice of metric on $\cU$.
\end{definition}

In practice, a blow-up model may be obtained by choosing a
riemannian metric to identify $\cN$ with the unit bundle
inside the geometric normal bundle $TB^\perp$ and then
by using the exponential map: $\beta (x,t) = \exp_p (tx)$ where $p$ is the projection onto $B$ of $x\in\cN$.

\begin{remark}
\label{rem:blowupmodel}
From the properties of $\beta_0$, it follows that:
\begin{itemize}
\item[(i)]
   the restriction of $\beta$ to $\cN \times (0, \varepsilon)$
is an orientation-preserving diffeomorphism onto $\cU \smallsetminus B$;
\item[(ii)]
   $\beta (-x,-t) = \beta (x,t)$;
\item[(iii)]
   the restriction of $\beta$ to $\cN \times \{ 0 \}$ is the bundle projection
$\cN \to B$;
\item[(iv)]
   for the vector fields $\nu$ generating the vertical bundle of
$\cN \to B$ and $\frac{\partial}{\partial t}$ tangent to
$(-\varepsilon, \varepsilon)$ we have that
$D \beta (\nu)$ intersects zero transversally and
$D \beta (\frac{\partial}{\partial t})$ is never zero.
\end{itemize}
\end{remark}

\begin{lemma}
If $\beta: \cN \times (-\varepsilon, \varepsilon) \to \cU$
is a blow-up model for the neighborhood $\cU$ of $B$ in $(M,\omega)$,
then the pull-back form $\beta^* \omega$ is an origami form
whose null foliation is the circle fibration
$\pi : \cN \times \{ 0 \} \to B$.
\end{lemma}

\begin{proof}
By properties~(i) and~(ii) in Remark~\ref{rem:blowupmodel}, the form 
$\beta^* \omega$ is symplectic away from $\cN \times \{ 0 \}$.
By property~(iii), on $\cN \times \{ 0 \}$
the kernel of $\beta^* \omega$ has dimension 2 and is fibrating.
By property~(iv) the top power of $\beta^* \omega$
intersects zero transversally.
\end{proof}

All blow-up models share the same germ up to diffeomorphism.
More precisely,
if $\beta_1: \cN \times (-\varepsilon, \varepsilon) \to \cU_1$
and $\beta_2: \cN \times (-\varepsilon, \varepsilon) \to \cU_2$
are two blow-up models for neighborhoods $\cU_1$ and $\cU_2$ of $B$
in $(M,\omega)$, then there are possibly narrower
tubular neighborhoods of $B$, $\cV_i \subseteq \cU_i$
and a diffeomorphism $\gamma: \cV_1 \to \cV_2$
such that $\beta_2 = \gamma \circ \beta_1$.
Moreover:

\begin{lemma}
\label{lem:isotopic_bupmodels}
Any two blow-up models
$\beta_i: \cN \times (-\varepsilon_i, \varepsilon_i) \to \cU_i$,
$i=1,2$, are isotopic, that is, can be smoothly connected by
a family of blow-up models.
\end{lemma}

\begin{proof}
By definition, the blow-up models factor as
\[
   \beta_i = \eta_i \circ \beta_0 \ , \qquad i=1,2 \ ,
\]
for some tubular neighborhood diffeomorphisms, $\eta_1$ and $\eta_2$,
which are isotopic since the set of riemannian metrics on $\cU$ is convex
and different geodesic flows are isotopic.
\end{proof}

Let $(M,\omega)$ be a symplectic manifold with a
codimension-two symplectic submanifold $B$.

\begin{definition}
A \textit{model involution} of a tubular neighborhood $\cU$ of $B$
is a symplectic involution $\gamma : \cU \to \cU$ preserving $B$
such that on the connected components $\cU_i$ of $\cU$ where
$\gamma (\cU_i) = \cU_i$ we have $\gamma |_{\cU_i} = \mbox{id}_{\cU_i}$.
\end{definition}

A model involution $\gamma$ induces a bundle involution
$\Gamma : \cN \to \cN$ covering $\gamma|_B$ by the formula
\[
   \Gamma [v] = \left[ d\gamma_p (v) \right]\ ,
   \qquad \mbox{ for } v \in T_pM, p \in B \ .
\]
This is well-defined because $\gamma (B) = B$.
We denote by $- \Gamma : \cN \to \cN$ the involution
$[v] \mapsto \left[ - d\gamma_p (v) \right]$.

\begin{remark}
\label{rmk:disjoint}
When $B$ is the disjoint union of $B_1$ and $B_2$,
and correspondingly $\cU = \cU_1 \sqcup \cU_2$,
if $\gamma (B_1) = B_2$ then
\[
   \gamma_1 := \gamma |_{\cU_1} : \cU_1 \to \cU_2
   \qquad \mbox{ and } \qquad
   \gamma |_{\cU_2} = \gamma_1 ^{-1}: \cU_2 \to \cU_1 \ .
\]
In this case, $B / \gamma \simeq B_1$
and $\cN / - \Gamma \simeq \cN_1$
is the radially projectivized normal bundle to $B_1$.
\end{remark}

\begin{proposition}
\label{prop:radial}
Let $(M,\omega)$ be a (compact) symplectic manifold,
$B$ a compact codimension-two symplectic submanifold
and $\cN$ its radially projectivized normal bundle.
Let $\gamma : \cU \to \cU$ be a model involution
of a tubular neighborhood $\cU$ of $B$ and
$\Gamma : \cN \to \cN$ the induced bundle map.

Then there is a (compact) origami manifold
$(\widetilde{M},\widetilde{\omega})$ with
symplectic part $\widetilde{M} \setminus Z$ symplectomorphic
to $M \setminus B$, folding
hypersurface diffeomorphic to $\cN / - \Gamma$
and null fibration isomorphic to
$\cN / - \Gamma \to B / \gamma$.
\end{proposition}

\begin{proof}
Choose $\beta: \cN \times (-\varepsilon, \varepsilon) \to \cU$
a blow-up model for the neighborhood $\cU$
such that $\gamma \circ \beta = \beta \circ \Gamma$.
This is always possible:
For components $\cU_i$ of $\cU$ where $\gamma (\cU_i) = \cU_i$
this condition is trivial;
for disjoint neighborhood components $\cU_i$ and $\cU_j$ such that
$\gamma (\cU_i) = \cU_j$ (as in Remark~\ref{rmk:disjoint}),
this condition amounts to choosing any blow-up model on one
of these components and transporting it to the other by $\gamma$.

Then $\beta^* \omega$ is a folded symplectic form on
$\cN \times (-\varepsilon, \varepsilon)$
with folding hypersurface $\cN \times \{ 0 \}$ and null foliation
integrating to the circle fibration
$S^1 \hookrightarrow \cN \stackrel{\pi}{\to} B$.
We define
\[
   \widetilde{M} = \left( M \smallsetminus B \bigsqcup
   \cN \times (-\varepsilon,\varepsilon) \right) \, / \, \sim
\]
where we quotient by
\[
   (x,t) \sim \beta (x,t) \mbox{ for } t>0
   \qquad \mbox{ and } \qquad
   (x,t) \sim \left( -\Gamma(x),-t \right)\ .
\]
The forms $\omega$ on $M \smallsetminus B$ and
$\beta^* \omega$ on $\cN \times (-\varepsilon,\varepsilon)$
induce on $\widetilde{M}$ an origami form $\widetilde{\omega}$
with folding hypersurface $\cN / - \Gamma$.
Indeed $\beta$ is a symplectomorphism for $t>0$,
and $\left( - \Gamma, -\mbox{id} \right)$
on $\cN \times (-\varepsilon,\varepsilon)$
is a symplectomorphism away from $t=0$
(since $\beta$ and $\gamma$ are)
and at points where $t=0$ it is a local diffeomorphism.
\end{proof}

\begin{definition}
An origami manifold $(\widetilde{M},\widetilde{\omega})$
as just constructed is called a \textit{radial blow-up of
$(M,\omega)$ through $(\gamma,B)$}.
\end{definition}

The next proposition states that radial blow-ups of
$(M,\omega)$ through $(\gamma,B)$ are unique up to symplectomorphism.

\begin{proposition}
Let $(M,\omega)$ be a symplectic manifold,
$B$ a compact codimension-two symplectic submanifold
and $\gamma : \cU \to \cU$ a model involution
of a tubular neighborhood $\cU$ of $B$.
Then different choices of a blow-up model for $\cU$
yield symplectomorphic radial blow-ups through $(\gamma,B)$.
\end{proposition}

\begin{proof}
Let $\beta_0$ and $\beta_1$ be two blow-up models for $\cU$
such that $\gamma \circ \beta_i = \beta_i \circ \Gamma$.
We restrict them to the same domain with $\varepsilon$
small enough and still denote
\[
   \beta_i : \cN \times (-\varepsilon,\varepsilon)
   \longrightarrow \cU_i \ , \qquad i=0,1 \ .
\]
Let $\beta_t : \cN \times (-\varepsilon,\varepsilon) \to \cU_t$
be a technical isotopy between $\beta_0$ and $\beta_1$
in the sense of~\cite[p.89]{br-ja:book}.
Each of the corresponding origami manifolds
$(\widetilde{M}_t,\widetilde{\omega}_t)$ is defined as a quotient of
the disjoint union of $(M \smallsetminus B,\omega)$
with $(\cN \times (-\varepsilon,\varepsilon),\beta_t^*\omega)$
by the equivalence relation $\sim_t$ given by $\beta_t$ and $\Gamma$
as in the proof of Proposition~\ref{prop:radial}.
Let $C$ be a compact neighborhood of
$\beta \left( [0,1] \times \cN \times [\delta,\varepsilon-\delta] \right) :=
\cup_{t \in [0,1]} \beta_t \left( \cN \times [\delta,\varepsilon-\delta] \right)$
in $M \setminus B$, where $\delta < \frac{\varepsilon}{2}$.

By Theorem~10.9 in~\cite{br-ja:book}, there is a smooth
family of diffeomorphisms
\[
   H_t : M \longrightarrow M \ , \qquad t \in [0,1] \ ,
\]
which hold fixed all points outside $C$ (in particular, the $H_t$ fix $B$),
with $H_0$ the identity map and such that
\[
   \beta_t |_{\cN \times [\delta,\varepsilon-\delta]}
   = H_t \circ \beta_0 |_{\cN \times [\delta,\varepsilon-\delta]} \ .
\]
By property~(ii) in Remark~\ref{rem:blowupmodel},
the same holds on $\cN \times [-\varepsilon+\delta,-\delta]$.

The diffeomorphism $H_t$ restricted to
$D := M \setminus \beta_0 \left( \cN \times (-\delta,\delta) \right)$
and the identity diffeomorphism on
$\cN \times (-\varepsilon + \delta, \varepsilon - \delta)$
together define a diffeomorphism
\[
   \phi_t : \widetilde{M}_0 \longrightarrow \widetilde{M}_t
\]
which fixes the fold $\cN / -\Gamma$.

All forms in the family $\phi_t^* \widetilde{\omega_t}$ on $\widetilde{M}_0$
are origami with the same fold, are equal at points of that fold,
and are all equal to $\widetilde{\omega_0}$ away from
\[
   C \cup \beta_0 \left( \cN \times [-\delta,\delta] \right)
\]
which is a compact neighborhood of the fold in $\widetilde{M}_0$
retracting to the fold.
Hence, by a folded version of Moser's trick
(see the proof of Theorem~1 in~\cite{ca-gu-wo:unfolding})
there is an isotopy $\rho_t : \widetilde{M}_0 \to \widetilde{M}_0$, $t \in \RR$,
fixing the fold such that
\[
   \rho_t^* \left( \phi_t^* \widetilde{\omega_t} \right) = \widetilde{\omega_0}
   \qquad \mbox{ for all } t \ .
\]
The map $\phi_1 \circ \rho_1$ is a symplectomorphism
between $(\widetilde{M}_0,\widetilde{\omega}_0)$ and
$(\widetilde{M}_1,\widetilde{\omega}_1)$.
\end{proof}

Different model involutions $\gamma_i : \cU_i \to \cU_i$
of tubular neighborhoods of $B$ give rise to symplectomorphic
origami manifolds as long as the induced bundle maps
$\Gamma_i : \cN \to \cN$ are isotopic.
Examples~\ref{ex:klein2} and~\ref{ex:klein} illustrate the
dependence on the model involution.

\begin{example}
\label{ex:rp2}
Let $M$ be a 2-sphere, $B$ one point on it,
and $\gamma$ the identity map on a neighborhood of that point.
Then a radial blow-up $\widetilde{M}$ is $\RR \PP^2$ and
$\widetilde{\omega}$ a form which folds along a circle.
\end{example}

\begin{example}
\label{ex:klein2}
Let $M$ be a 2-sphere, $B$ the union of two (distinct) points on it,
and $\gamma$ the identity map on a neighborhood of those points.
Then a radial blow-up $\widetilde{M}$ is a Klein bottle
and $\widetilde{\omega}$ a form which folds along two circles.
\end{example}

\begin{example}
\label{ex:klein}
Again, let $M$ be a 2-sphere, $B$ the union of two (distinct) points on it,
and now $\gamma$ defined by a symplectomorphism from a Darboux neighborhood
of one point to a Darboux neighborhood of the other.
Then a radial blow-up $\widetilde{M}$ is a Klein bottle
and $\widetilde{\omega}$ a form which folds along a circle.
\end{example}

\begin{remark}
\label{rmk:tubular}
The quotient $\cN \times (-\varepsilon,\varepsilon) /
\left( -\Gamma, -\mbox{id} \right)$
provides a collar neighborhood of the fold in
$(\widetilde{M},\widetilde{\omega})$.

When $B$ splits into two disjoint components
interchanged by $\gamma$ as in Remark~\ref{rmk:disjoint},
this collar is orientable so the fold is coorientable.
Example~\ref{ex:klein}
illustrates a case where, even though the fold is coorientable,
the radial blow-up $(\widetilde{M},\widetilde{\omega})$
is not orientable.

When $\gamma$ is the identity map, as in Example~\ref{ex:rp2},
the collar is nonorientable and the fold is not coorientable.
In the latter case, the collar is a bundle of M\"obius bands
$S^1 \times (-\varepsilon,\varepsilon) / (-\mbox{id},-\mbox{id})$ over $B$.

In general, $\gamma$ will be the identity over some
connected components of $B$ and will interchange other components,
so some components of the fold will be coorientable
and others will not.
\end{remark}

\begin{remark}
\label{rmk:orientable}
For the radial blow-up $(\widetilde{M},\widetilde{\omega})$ to be orientable,
the starting manifold $(M,\omega)$ must be the disjoint union
of symplectic manifolds $(M_1,\omega_1)$ and $(M_2,\omega_2)$
with $B = B_1 \cup B_2$, $B_i \subset M_i$, such that
$\gamma (B_1) = B_2$ as in Remark~\ref{rmk:disjoint}.
In this case $(\widetilde{M},\widetilde{\omega})$
may be equipped with an orientation such that
\[
   \widetilde{M}^+ \simeq M_1 \smallsetminus B_1
   \qquad \mbox{ and } \qquad
   \widetilde{M}^- \simeq M_2 \smallsetminus B_2 \ .
\]
The folding hypersurface is diffeomorphic to $\cN_1$ (or $\cN_2$)
and we have
\[
   \widetilde{\omega}\simeq\left\{
   \begin{array}{ll}
   \omega_1&\text{ on } M_1 \smallsetminus B_1\\ 
   \omega_2&\text{ on } M_2 \smallsetminus B_2\\
   \beta^* \omega_1&\text{ on }  \cN \times (-\varepsilon,\varepsilon)
\end{array}\right.
\]
We then say that {\em $(\widetilde{M},\widetilde{\omega})$
is the blow-up of $(M_1,\omega_1)$ and $(M_2,\omega_2)$
through $(\gamma_1,B_1)$} where $\gamma_1$ is the restriction of
$\gamma$ to a tubular neighborhood of $B_1$.
\end{remark}

\begin{remark}
Radial blow-up may be performed on an origami manifold
at a symplectic submanifold $B$ (away from the fold).
When we start with two folded surfaces and radially blow them up
at one point (away from the folding curves),
topologically the resulting manifold is a
connected sum at a point, $M_1 \# \overline{M_2}$,
with all the previous folding curves plus a new closed
curve.
Since all
$\RR \PP ^{2n}$ are folded symplectic manifolds, the standard real
blow-up of a folded symplectic manifold at a point away from its
folding hypersurface still admits a folded symplectic form,
obtained by viewing this operation as a connected sum.
\end{remark}

\subsection{Cutting a radial blow-up}
\label{ss:cutting_blowup}

\begin{proposition}
\label{prop:cutting_blowup}
Let $(M,\omega)$ be a radial blow-up of
the symplectic manifolds
$(M_1,\omega_1)$ and $(M_2,\omega_2)$
through $(\gamma_1,B_1)$ where $\gamma_1$
is a symplectomorphism of tubular neighborhoods
of codimension-two symplectic submanifolds $B_1$ and $B_2$
of $M_1$ and $M_2$, respectively, taking $B_1$ to $B_2$.

Then cutting $(M,\omega)$ yields manifolds
symplectomorphic to $(M_1,\omega_1)$ and $(M_2,\omega_2)$
where the symplectomorphisms carry $B$ to $B_1$ and $B_2$.
\end{proposition}

\begin{proof}
We first exhibit a symplectomorphism $\rho_1$ between a
cut space $(M_0^+,\omega_0^+)$ of $(M,\omega)$
and the original manifold $(M_1,\omega_1)$.

Let $\cN$ be the radially projectivized normal bundle to $B_1$ in $M_1$
and let $\beta : \cN \times (-\varepsilon,\varepsilon) \to \cU_1$
be a blow-up model.
The cut space $M_0^+$ is obtained gluing the reduced space
\[
   \mu^{-1} (0) / S^1 =
   \left\{ (x,s,z) \in Z \times [0,\varepsilon^2) \times \CC
   \; | \; s = \textstyle{\frac{|z|^2}{2}} \right\} / S^1
\]
with the manifold
\[
   M_1 \smallsetminus B_1
\]
via the diffeomorphisms
\[
\begin{array}{rclcrcl}
   \cN \times (0,\varepsilon) & \longrightarrow & \mu^{-1} (0) / S^1
   & \quad \mbox{and} \quad &
   \cN \times (0,\varepsilon) & \longrightarrow & \cU_1 \smallsetminus B_1 \\
   (x,t) & \longmapsto & [x,t^2,t\sqrt{2}]
   & & (x,t) & \longmapsto & \beta (x,t)
\end{array}
\]
i.e., the gluing is by the identification
$[x,t^2,t\sqrt{2}] \sim \beta (x,t)$ for $t>0$ over $\cU_1 \smallsetminus B_1$.
The symplectic form $\omega_0^+$ on $M_0^+$ is equal to the reduced
symplectic form on $\mu^{-1} (0) / S^1$ and equal to $\omega_1$
on $M_1 \smallsetminus B_1$ (the gluing diffeomorphism
$[x,t^2,t\sqrt{2}] \mapsto \beta (x,t)$ is a symplectomorphism).

We want to define a map $\rho_1 : M_1 \to M_0^+$ which is
the identity on $M_1 \smallsetminus B_1$ and on $\cU_1$ is the
composed diffeomorphism
\[
\begin{array}{rrcccl}
   \delta_1 : & \cU_1 & \longrightarrow & \left( \cN \times \CC \right) / S^1
   & \longrightarrow & \mu^{-1} (0) / S^1 \\
   & & & [x,z] & \longmapsto & [x,|z|^2, z\sqrt{2}]
\end{array}
\]
where the first arrow is the inverse of the bundle isomorphism
given by the blow-up model.
In order to show that $\rho_1$ is well-defined
we need to verify that $u_1 \in \cU_1 \smallsetminus B_1$
is equivalent to its image
$\delta_1 (u_1) \in \mu^{-1} (0) / S^1 \smallsetminus B$.
Indeed $u_1$ must correspond to
$[x,z] \in \left( \cN \times \CC \right) / S^1$ with $z \neq 0$.
We write $z$ as $z=te^{i\theta}$ with $t>0$.
Since $[x,z] = [e^{i\theta}x,t]$, we have $u_1 = \beta(e^{i\theta}x,t)$
and $\delta_1 (u_1) = [e^{i\theta}x,|t|^2, t\sqrt{2}]$.
These two are equivalent under $\beta (x,t) \sim [x,t^2,t\sqrt{2}]$,
so $\rho_1$ is well-defined.

Furthermore, $M_1$ and $M_0^+$ are symplectic manifolds
equipped with a diffeomorphism which is a symplectomorphism
on the common dense subset $M_1 \smallsetminus B_1$.
We conclude that $M_1$ and $M_0^+$ must be globally symplectomorphic.

Now we tackle $(M_2,\omega_2)$ and $(M_0^-,\omega_0^-)$.
The cut space $M_0^-$ is obtained gluing the same reduced space
$\mu^{-1} (0) / S^1$ with the manifold  $M_2 \smallsetminus B_2$
via the diffeomorphisms
\[
\begin{array}{rcl}
   \cN \times (-\varepsilon,0) & \longrightarrow
   & \overline{\mu^{-1} (0) / S^1} \\
   (x,t) & \longmapsto & [x,t^2,t\sqrt{2}]
\end{array}
\]
and
\[
\begin{array}{rcccl}
   \cN \times (-\varepsilon,0) & \longrightarrow
   & \overline{\cU_1 \smallsetminus B_1} & \stackrel{\gamma}{\longrightarrow}
   & \overline{\cU_2 \smallsetminus B_2} \\
   (x,t) & \longmapsto & \beta (x,t) & \longmapsto
   & \gamma \left( \beta (x,t) \right)
\end{array}
\]
i.e., the gluing is by the identification
$[x,t^2,t\sqrt{2}] \sim \gamma \left( \beta (x,t) \right)$
for $t < 0$ over $\cU_2 \smallsetminus B_2$.
The symplectic form $\omega_0^-$ on $M_0^-$ restricts
to the reduced form on $\mu^{-1} (0) / S^1$
and $\omega_2$ on $M_2 \smallsetminus B_2$.

We want to define a map $\rho_2 : M_2 \to M_0^-$ as being
the identity on $M_2 \smallsetminus B_2$ and on $\cU_2$ being the
composed diffeomorphism
\[
\begin{array}{rrcccrcl}
   \delta_2 : & \cU_2 & \stackrel{\gamma^{-1}}{\longrightarrow}
   & \cU_1 & \longrightarrow & \left( \cN \times \CC \right) / S^1
   & \longrightarrow & \mu^{-1} (0) / S^1 \\
   & & & & & [x,z] & \longmapsto & [x,|z|^2, z\sqrt{2}]
\end{array}
\]
where the second arrow is the inverse of the bundle isomorphism
given by the blow-up model.
In order to show that $\rho_2$ is well-defined
we need to verify that $u_2 = \gamma( u_1 ) \in \cU_2 \smallsetminus B_2$
is equivalent to its image
$\delta_2 (u_2) \in \mu^{-1} (0) / S^1 \smallsetminus B$.
Indeed $u_1$ must correspond to
$[x,z] \in \left( \cN \times \CC \right) / S^1$ with $z \neq 0$.
We write $z$ as $z=-te^{i\theta}$ with $t<0$.
From $[x,z] = [-e^{i\theta}x,t]$, we conclude that
$u_2 = \gamma \left( \beta(-e^{i\theta}x,t) \right)$
and $\delta_2 (u_2) = [-e^{i\theta}x,|t|^2, t\sqrt{2}]$.
These two are equivalent under
$\gamma \left( \beta (x,t) \right) \sim [x,t^2,t\sqrt{2}]$,
so $\rho_2$ is well-defined.

As before, we conclude that $M_2$ and $M_0^-$
must be globally symplectomorphic.
\end{proof}

\begin{lemma}
\label{lemma:double_cover}
Let $(M,\omega)$ be a radial blow-up of the symplectic manifold
$(M_s,\omega_s)$ through $(\gamma,B)$.
We write $B = B_0 \sqcup B_1 \sqcup B_2$ and the domain of $\gamma$
as $\cU = \cU_0 \sqcup \cU_1 \sqcup \cU_2$ where $\gamma$ is the
identity map on $\cU_0$ and exchanges $\cU_1$ and $\cU_2$.

Let $(\overline{M_s},\overline{\omega_s})$ be the trivial double cover
of $(M_s,\omega_s)$ with
$\overline{B} = B^{\uparrow} \sqcup B^{\downarrow}$,
$\overline{\cU} = \cU^{\uparrow} \sqcup \cU^{\downarrow}$
the double covers of $B$ and $\cU$.
Let $\overline{\gamma} : \overline{\cU} \to \overline{\cU}$
be the lift of $\gamma$ satisfying
\[
   \overline{\gamma} (\cU_0^{\uparrow}) = \cU_0^{\downarrow} \ ,
   \qquad
   \overline{\gamma} (\cU_1^{\uparrow}) = \cU_2^{\downarrow}
   \qquad \mbox{ and } \qquad
   \overline{\gamma} (\cU_2^{\uparrow}) = \cU_1^{\downarrow}
\]
and let $(\overline{M},\overline{\omega})$ be a radial blow-up of
$(\overline{M_s},\overline{\omega_s})$ through
$(\overline{\gamma}, \overline{B})$.

Then $(\overline{M},\overline{\omega})$ is an orientable
double cover of $(M,\omega)$.
\end{lemma}

\begin{proof}
Since it is the double cover of an oriented manifold, we write
\[
   \overline{M_s} = M_s^{\uparrow} \sqcup M_s^{\downarrow}
\]
with each component diffeomorphic to $M_s$.
By Remark~\ref{rmk:orientable}, the blow-up
$(\overline{M},\overline{\omega})$ is orientable and has
\[
   \overline{M}^+ \simeq  M_s^{\uparrow} \smallsetminus B^{\uparrow}
   \qquad \mbox{ and } \qquad
   \overline{M}^- \simeq  M_s^{\downarrow} \smallsetminus B^{\downarrow}
\]
with fold $\cN^{\uparrow}$ fibering over $B^{\uparrow}$.
There is a natural two-to-one smooth projection
$\overline{M} \to M$ taking
$M_s^{\uparrow} \smallsetminus B^{\uparrow}$ and
$M_s^{\downarrow} \smallsetminus B^{\downarrow}$
each diffeomorphically to $M \smallsetminus Z$ where $Z$ is
the fold of $(M,\omega)$, and taking the fold $\cN^{\uparrow} \simeq \cN$
of $(\overline{M},\overline{\omega})$ to $Z \simeq \cN / -\Gamma$
with $\Gamma : \cN \to \cN$ the bundle map induced by $\gamma$
(the map $-\Gamma$ having no fixed points).
\end{proof}

\begin{corollary}
\label{coroll:cutting_blowup}
Let $(M,\omega)$ be a radial blow-up of
the symplectic manifold $(M_s,\omega_s)$ through $(\gamma,B)$.
Then cutting $(M,\omega)$ yields a manifold
symplectomorphic to $(M_s,\omega_s)$
where the symplectomorphism carries the base to $B$.
\end{corollary}

\begin{proof}
Let $(M_{cut},\omega_{cut})$ be a symplectic cut space of $(M,\omega)$.
Let $(\overline{M_s},\overline{\omega_s})$ and
$(\overline{M_{cut}},\overline{\omega_{cut}})$ be the
trivial double covers of $(M_s,\omega_s)$
and $(M_{cut},\omega_{cut})$.
By Lemma~\ref{lemma:double_cover}, the radial blow-up
$(\overline{M},\overline{\omega})$ of 
$(\overline{M_s},\overline{\omega_s})$ through
$(\overline{\gamma},\overline{B})$ is an orientable double cover
of $(M,\omega)$.
As a consequence of Definition~\ref{defn:cut_space},
$(\overline{M_{cut}},\overline{\omega_{cut}})$
is the symplectic cut space of $(\overline{M},\overline{\omega})$.
By Proposition~\ref{prop:cutting_blowup},
$(\overline{M_s},\overline{\omega_s})$ and
$(\overline{M_{cut}},\overline{\omega_{cut}})$
are symplectomorphic relative to the centers.
It follows that $(M_s,\omega_s)$ and $(M_{cut},\omega_{cut})$
are symplectomorphic relative to the centers.
\end{proof}

\subsection{Radially blowing-up cut pieces}
\label{ss:blowup_cut}

\begin{proposition}
\label{prop:blowingup_unfold}
Let $(M,\omega)$ be an oriented origami manifold
with null fibration $Z \stackrel{\pi}{\to} B$.

Let $(M_1,\omega_1)$ and $(M_2,\omega_2)$ be its symplectic cut pieces,
$B_1$ and $B_2$ the natural symplectic embedded images of $B$ in each
and $\gamma_1 : \cU_1 \to \cU_2$ the symplectomorphism
of tubular neighborhoods of $B_1$ and $B_2$ as in
Remark~\ref{rmk:involution}.

Let $(\widetilde{M},\widetilde{\omega})$ be a radial blow-up
of $(M_1,\omega_1)$ and $(M_2,\omega_2)$ through $(\gamma_1 , B_1)$.

Then $(M,\omega)$ and $(\widetilde{M},\widetilde{\omega})$
are equivalent origami manifolds.
\end{proposition}

\begin{proof}
Choose a principal $S^1$ action
$S^1 \hookrightarrow Z \stackrel{\pi}{\to} B$
and let $\varphi : Z \times (-\varepsilon,\varepsilon) \to \cU$
be a {\em Moser model} for a tubular neighborhood $\cU$ of $Z$
in $M$ as in the proof of Proposition~\ref{pro:cutting}.
Let $\cN$ be the radial projectivized normal bundle to $B_1$ in $M_1$.
By Proposition~\ref{pro:cutting},
the natural embedding of $B$ in $M_1$ with image $B_1$
lifts to a bundle isomorphism from $\cN \to B_1$ to $Z \to B$.
Under this isomorphism, we pick the following
blow-up model for the neighborhood $\mu^{-1}(0) / S^1$ of $B_1$
in $(M_1,\omega_1)$:
\[
\begin{array}{rccl}
   \beta : & Z \times (- \varepsilon , \varepsilon)
   & \longrightarrow & \mu^{-1}(0) / S^1 \\
   & (x,t) & \longmapsto & [x,t^2,t\sqrt{2}] \ .
\end{array}
\]
By recalling the construction of the reduced form $\omega_1$
on $\mu^{-1}(0) / S^1$ (see proof of  Proposition~\ref{pro:cutting})
we find that $\beta^* \omega_1 = \varphi^* \omega$.
Hence in this case the origami manifold
$(\widetilde{M},\widetilde{\omega})$ has
\[
   \widetilde{M} = \left( M_1 \smallsetminus B_1 \bigcup
   \overline{M_2 \smallsetminus B_2} \bigcup
   Z \times (-\varepsilon,\varepsilon) \right) \, / \, \sim
\]
where we quotient identifying
\[
   Z \times (0,\varepsilon) \stackrel{\beta}{\simeq}
   \mu^{-1}(0) / S^1 \subset \cU_1 \smallsetminus B_1
\]
and
\[
   Z \times (-\varepsilon,0) \stackrel{\beta}{\simeq}
   \overline{\mu^{-1}(0) / S^1} \stackrel{\gamma}{\simeq} 
   \overline{\mu^{-1}(0) / S^1} \subset \overline{\cU_2 \smallsetminus B_2}
\]
and we have
\[
   \widetilde{\omega}:=\left\{
   \begin{array}{ll}
   \omega_1&\text{ on } M_1 \smallsetminus B_1\\ 
   \omega_2&\text{ on } M_2 \smallsetminus B_2\\
   \beta^* \omega_1&\text{ on }  Z \times (-\varepsilon,\varepsilon) \ .
   \end{array}\right.
\]
The symplectomorphisms (from the proof of Proposition~\ref{pro:cutting})
$\overline{j^+} : M^+ \to M_1 \smallsetminus B_1$
and 
$\overline{j^-} : M^- \to M_2 \smallsetminus B_2$
extending $\varphi(x,t)\mapsto\left[x,t^2,t\sqrt{2}\right]$
make the following diagrams (one for $t>0$, the other for $t<0$) commute:
\[
\begin{array}{rcl}
   M^+ \supset \cU^+
   & \stackrel{j^+}{\longrightarrow} &
   \mu^{-1}(0) / S^1 \subset M_1 \smallsetminus B_1 \\
   \varphi \nwarrow & & \nearrow \beta \\
   & Z \times (0,\varepsilon)
\end{array}
\]
and
\[
\begin{array}{rcl}
   M^- \supset \cU^-
   & \stackrel{j^-}{\longrightarrow} &
   \overline{\mu^{-1}(0) / S^1}
   \subset \overline{M_2 \smallsetminus B_2} \\
   \varphi \nwarrow & & \nearrow \beta \\
   & Z \times (-\varepsilon,0)
\end{array}
\]
Therefore, the maps $\overline{j^+}$, $\overline{j^-}$ and $\varphi^{-1}$
together define a diffeomorphism from $M$ to $\widetilde{M}$
pulling back $\widetilde{\omega}$ to $\omega$.
\end{proof}

\begin{corollary}
\label{coroll:blowingup_unfold}
Let $(M,\omega)$ be an origami manifold
with null fibration $Z \stackrel{\pi}{\to} B$.

Let $(M_{cut},\omega_{cut})$ be its symplectic cut space,
$B_{cut}$ the natural symplectic embedded image of $B$ in $M_{cut}$
and $\gamma : \cU \to \cU$ a symplectomorphism
of a tubular neighborhood $\cU$ of $B_{cut}$ as in Remark~\ref{rmk:involution}.

Let $(\widetilde{M},\widetilde{\omega})$ be a radial blow-up
of $(M_{cut},\omega_{cut})$ through $(\gamma,B_{cut})$.

Then $(M,\omega)$ and $(\widetilde{M},\widetilde{\omega})$
are symplectomorphic origami manifolds.
\end{corollary}

\begin{proof}
We pass to the orientable double covers.
By Proposition~\ref{prop:blowingup_unfold}, the orientable
double cover of $(M,\omega)$ is symplectomorphic to the blow-up
of its cut space.
By definition, the cut space of the double cover
of $(M,\omega)$ is the double cover of $(M_{cut},\omega_{cut})$.
By Lemma~\ref{lemma:double_cover}, the blow-up of the latter
double cover is the double cover of $(\widetilde{M},\widetilde{\omega})$.
\end{proof}


\section{Origami Polytopes}
\label{sec:polytopes}


\subsection{Origami convexity}
\label{sec:convexity}

\begin{definition}
\label{defn:agrees}
If $F_i$ is a face of a polytope $\Delta_i$, $i=1,2$, in $\RR^n$,
we say that $\Delta_1$ near $F_1$ \textit{agrees} with $\Delta_2$ near $F_2$
when $F_1 = F_2$ and there is an open subset $\cU$ of $\RR^n$
containing $F_1$ such that $\cU \cap \Delta_1 = \cU \cap \Delta_2$.
\end{definition}

The following is an origami analogue of the
Atiyah-Guillemin-Sternberg convexity theorem.

\begin{theorem}
\label{thm:convexity}
Let $(M,\omega,G,\mu)$ be a connected compact origami manifold
with null fibration $Z \stackrel{\pi}{\to} B$
and a hamiltonian action of an $m$-dimensional torus $G$
with moment map $\mu : M \to \fg^*$.
Then:

\begin{itemize}
\item[(a)]
The image $\mu (M)$ of the moment map is the union of
a finite number of convex polytopes $\Delta_i$, $i=1,\ldots,N$,
each of which is the image of the moment map restricted
to the closure of a connected component of $M \smallsetminus Z$.

\item[(b)]
Over each connected component $Z'$ of $Z$,
the null fibration is given by a subgroup of $G$
if and only if $\mu(Z')$
is a facet of each of the one or two polytopes corresponding to the
neighboring component(s) of $M \smallsetminus Z$, and when those are
two polytopes, they agree near the facet $\mu (Z')$.
\end{itemize}
\end{theorem}

We call such images $\mu (M)$ \textit{origami polytopes}.

\begin{remark}
When $M$ is oriented, the facets from part~(b) are always
shared by {\em two} polytopes.
In general, a component $Z'$ is coorientable if and only if
$\mu (Z')$ is a facet of two polytopes.
\end{remark}

\begin{proof}
\begin{itemize}
\item[(a)]
Since the $G$-action preserves $\omega$, it also preserves each 
connected component of the folding hypersurface $Z$
and its null foliation $V$ .
Choose an oriented trivializing section $w$ of $V$.
Average $w$ so that it is $G$-invariant, i.e., replace it with
\[
   \frac{1}{|G|} \int_G g_* \left( w_{g^{-1}(p)} \right) \, dg \ .
\]
Next, scale it uniformly over each orbit
so that its integral curves all have period $2 \pi$,
producing a vector field $v$ which generates an action of 
$S^1$ on $Z$ that commutes with the $G$-action.
This $S^1$-action also preserves the moment map $\mu$: 
for any $X\in\fg$ with corresponding 
vector field $X^\#$ on $M$, we have over $Z$
\[
   \cL_v \langle \mu , X \rangle
   = \imath_v d \langle \mu , X \rangle
   = - \imath_v \imath_{X^\#} \omega = \omega (v,X^\#) = 0 \ .
\]

Using this $v$, the cutting construction from
Section~\ref{sec:origami} has a hamiltonian version.
Let $(M_i,\omega_i)$, $i=1,\ldots,N$,
be the resulting compact connected components
of the symplectic cut space.
Let $B_i$ be the union of the components of the base $B$
which naturally embed in $M_i$.
Each $M_i \smallsetminus B_i$ is symplectomorphic to a connected 
component $\cW_i\subset M \smallsetminus Z$ and
$M_i$ is the closure of $M_i \smallsetminus B_i$.
Each $(M_i,\omega_i)$ inherits a hamiltonian action of $G$
with moment map $\mu_i$ which matches $\mu|_{\cW_i}$ over 
$M_i \smallsetminus B_i$ and is the well-defined $S^1$-quotient 
of $\mu|_Z$ over $B_i$.

By the Atiyah-Guillemin-Sternberg convexity
theorem~\cite{at:convexity,gu-st:convexity}, each 
$\mu_i (M_i)$ is a convex polytope $\Delta_i$.
Since $\mu (M)$ is the union of the $\mu_i (M_i)$, 
we conclude that
\[
   \mu (M) = \bigcup_{i=1}^N \Delta_i \ .
\]

\item[(b)]
Assume first that $M$ is orientable.

Let $Z'$ be a connected component of $Z$ with null fibration
$Z' \to B'$.
Let $\cW_1$ and $\cW_2$ be the two neighboring
components of $M \smallsetminus Z$ on each side of $Z'$,
$(M_1,\omega_1,G,\mu_1)$ and $(M_2,\omega_2,G,\mu_2)$
the corresponding cut spaces with moment polytopes
$\Delta_1$ and $\Delta_2$.

Let $\cU$ be a $G$-invariant tubular neighborhood of $Z'$ with
a $G$-equivariant diffeomorphism
$\varphi : Z' \times (-\varepsilon,\varepsilon) \to \cU$ such that
\[
   \varphi^* \omega = p^* i^* \omega
   + d \left( t^2 p^* \alpha \right) \ ,
\]
where $G$ acts trivially on $(-\varepsilon,\varepsilon)$,
$p: Z' \times (-\varepsilon,\varepsilon) \to Z'$
is the projection onto the first factor,
$t$ is the real coordinate on the interval $(-\varepsilon,\varepsilon)$
and $\alpha$ is a $G$-invariant $S^1$-connection on $Z'$
for a chosen principal $S^1$ action,
$S^1 \hookrightarrow Z \stackrel{\pi}{\to} B$.
The existence of such $\varphi$
follows from an equivariant Moser trick,
analogous to that in the proof of Proposition~\ref{pro:cutting}.

Without loss of generality,
$Z' \times (0,\varepsilon)$ and $Z' \times (-\varepsilon,0)$
correspond via $\varphi$ to the two sides $\cU_1 := \cU \cap \cW_1$
and $\cU_2 := \cU \cap \cW_2$, respectively.
The involution $\tau : \cU \to \cU$ translating
$t \mapsto -t$ in  $Z' \times (-\varepsilon,\varepsilon)$
is a $G$-equivariant (orientation-reversing) diffeomorphism preserving $Z'$,
switching $\cU_1$ and $\cU_2$ but preserving $\omega$.
Hence the moment map satisfies $\mu \circ \tau = \mu$
and $\mu (\cU_1) = \mu (\cU_2)$.

When the null fibration is given by a subgroup of $G$,
we cut the $G$-space $\cU$ at the level $Z'$.
The image $\mu(Z')$ is the intersection of $\mu(\cU)$ with a hyperplane
and thus a facet of both $\Delta_1$ and $\Delta_2$.

Each $\cU_i \cup B'$ is equivariantly symplectomorphic to
a neighborhood $\cV_i$ of $B'$ in $(M_i,\omega_i,G,\mu_i)$
with $\mu_i (\cV_i) = \mu(\cU_i) \cup \mu (Z')$, $i=1,2$.
As a map to its image, the moment map is open~\cite{ka-ma:convexity}.
Since $\mu_1 (\cV_1) = \mu_2 (\cV_2)$, we conclude that
$\Delta_1$ and $\Delta_2$ agree near the facet $\mu(Z')$.

For general null fibration, we cut the $G \times S^1$-space $\cU$
with moment map $(\mu, t^2)$ at $Z'$, the $S^1$-level $t^2 = 0$.
The image of $Z'$ by the $G \times S^1$-moment map
is the intersection of the image of the full $\cU$ with a hyperplane.
We conclude that the image $\mu (Z')$ is the first factor projection
$\pi : \fg^* \times \RR \to \fg^*$ of a facet of
a polytope $\widetilde \Delta$
in $\fg^* \times \RR$, so it can be of codimension zero or one;
see Example~\ref{ex:notfacet}.

If $\pi|_{\widetilde \Delta}:\widetilde \Delta \to \Delta_1$ is one-to-one,
then facets of $\widetilde \Delta$ map to facets of $\Delta_1$
and $\widetilde \Delta$ is contained in a hyperplane
surjecting onto $\fg^*$.
The normal to that hyperplane corresponds to a circle subgroup
of $G \times S^1$ acting trivially on $\cU$ and surjecting onto the
$S^1$-factor.
This allows us to express the $S^1$-action in terms of a subgroup of $G$.

If $\pi|_{\widetilde \Delta}:\widetilde \Delta \to \Delta_1$
is not one-to-one, it cannot map the facet $\widetilde F_{Z'}$ of 
$\widetilde \Delta$ corresponding to $Z'$ to a facet of $\Delta_1$:
Otherwise, $\widetilde F_{Z'}$ would contain nontrivial 
{\em vertical} vectors $(0,x)\in\fg^* \times \RR$ 
which would contradict the fact that the $S^1$ direction is that
of the null fibration on $Z'$.
Hence, the normal to $\widetilde F_{Z'}$ in $\widetilde \Delta$
must be transverse to $\fg^*$,
and the corresponding null fibration circle subgroup is not
a subgroup of $G$.

When $M$ is not necessarily orientable, we consider its
orientable double cover and lift the hamiltonian torus action.
The lifted moment map is the composition of the two-to-one projection
with the original double map, and the result follows.
\end{itemize}
\end{proof}

\begin{example}
\label{ex:toric_sphere}
Consider $(S^4, \omega_0, \TT^2, \mu)$ where
$(S^4, \omega_0)$ is a sphere as in Example~\ref{ex:spheres} 
with $\TT^2$ acting by
\[
   (e^{i\theta_1},e^{i\theta_2}) \cdot
   \underbrace{(z_1,z_2,h)}_{\in \CC^2 \times \RR \simeq \RR^5}
   = (e^{i\theta_1} z_1 ,e^{i\theta_2} z_2 ,h)
\]
and moment map defined by
\[
   \mu (z_1,z_2,h) = \Big( \textstyle{\underbrace{\frac{|z_1|^2}{2}}_{x_1} ,
   \underbrace{\frac{|z_2|^2}{2}}_{x_2}} \Big)
\]
whose image is the triangle
$x_1 \geq 0$, $x_2 \geq 0$, $x_1+x_2 \leq \frac 12$.
The image $\mu (Z)$ of the folding hypersurface (the equator)
is the hypotenuse.

\begin{figure}[ht]
\begin{center}
\includegraphics[scale=.6]{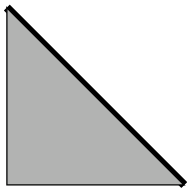}
\caption{An origami polytope for a 4-sphere}
\label{fig:plain_triangle}
\end{center}
\end{figure}
The null foliation is the Hopf fibration given by the
diagonal circle subgroup of $\TT^2$.
In this case, Theorem~\ref{thm:convexity} says
that the triangle is the union of two identical
triangles, each of which is the moment polytope
of one of the $\CC \PP^2$'s obtained by cutting;
see Example~\ref{ex:cutting_spheres}.
Likewise, if $(S^4, \omega_0, \TT^2, \mu)$ was blown-up
at a pole, the triangle in Figure~\ref{fig:plain_triangle} would be the 
superposition of the same triangle with a trapezoid.
\end{example}

\begin{example}
\label{ex:notfacet}
Consider $(S^2 \times S^2, \omega_s \oplus \omega_f, S^1, \mu)$,
where $(S^2 , \omega_s)$ is a standard symplectic sphere,
$(S^2 , \omega_{f})$ is a folded symplectic sphere
with folding hypersurface given by a parallel,
and $S^1$ acts as the diagonal of the standard rotation action of
$S^1 \times S^1$ on the product manifold.
Then the moment map image is a line segment and the image of the
folding hypersurface is a nontrivial subsegment.
Indeed, the image of $\mu$ is a $45^o$ projection of the image
of the moment map for the full $S^1 \times S^1$ action, i.e.,
a rectangle in which the folding hypersurface surjects
to one of the sides; see Figure~\ref{fig:projection}.

\begin{figure}[ht]
\begin{center}
\includegraphics[scale=.3]{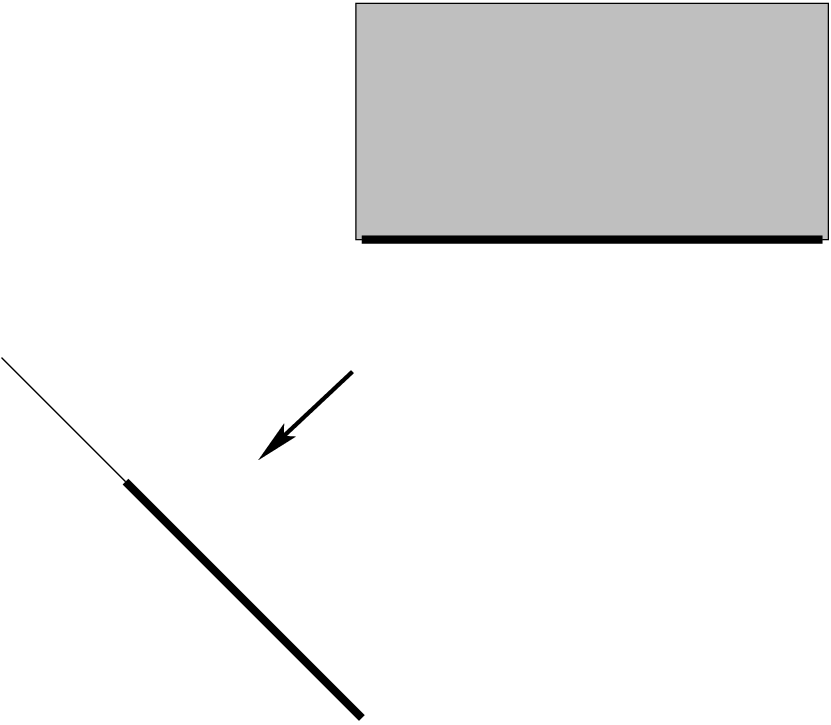}
\caption{Origami polytopes for a product of two 2-spheres}
\label{fig:projection}
\end{center}
\end{figure}
By considering the first or second factors of $S^1 \times S^1$ alone,
we get the two extreme cases in which the image of the folding hypersurface
is either the full line segment or simply one of the boundary points.

The analogous six-dimensional examples
$(S^2 \times S^2 \times S^2,
\omega_s \oplus \omega_s \oplus \omega_f, \TT^2 , \mu)$
produce moment images which are rational projections of a cube, with
the folding hypersurface mapped to rhombi; see Figure~\ref{fig:rhombus}.

\begin{figure}[ht]
\begin{center}
\includegraphics[scale=.24]{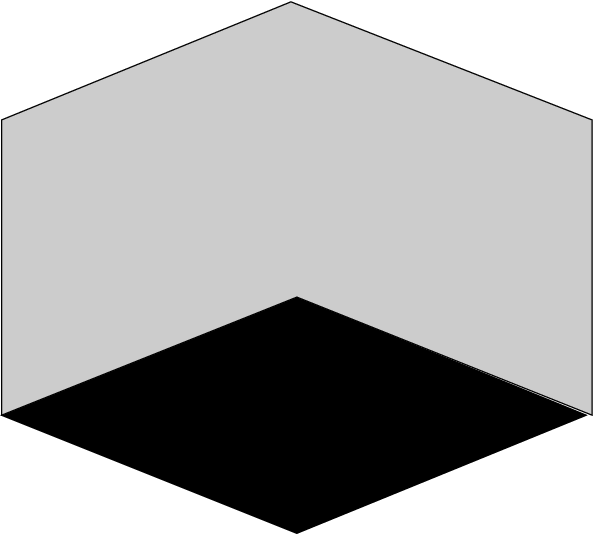}
\caption{An origami polytope for a product of three 2-spheres}
\label{fig:rhombus}
\end{center}
\end{figure}
\end{example}


\subsection{Toric case}
\label{ss:toric}

\begin{definition}
A \textit{toric origami manifold} $(M,\omega,G,\mu)$
is a compact connected origami manifold
$(M,\omega)$ equipped with an effective hamiltonian action of a
torus $G$ with $\dim G = \frac 12 \dim M$
and with a choice of a corresponding moment map $\mu$.
\end{definition}

For a toric origami manifold $(M,\omega,G,\mu)$,
orbits with trivial isotropy -- the {\em principal orbits} --
form a dense open subset of $M$~\cite[p.179]{br:transformation}.
Any {\em coorientable} connected component $Z'$ of $Z$
has a $G$-invariant tubular neighborhood
modelled on $Z' \times (-\varepsilon,\varepsilon)$ with a
$G \times S^1$ hamiltonian action having moment map $(\mu, t^2)$.
As the orbits are isotropic submanifolds, the principal orbits
of the $G \times S^1$-action must still have dimension $\dim G$.
Their stabilizer must be a one-dimensional compact
connected subgroup surjecting onto $S^1$.
Hence, over those connected components of $Z$,
the null fibration is given by a subgroup of $G$.
A similar argument holds for {\em noncoorientable} connected 
components of $Z$, using orientable double covers.
We have thus proven the following corollary of Theorem~\ref{thm:convexity}.

\begin{corollary}
\label{cor:agreeing}
When $(M,\omega,G,\mu)$ is a toric origami manifold,
the moment map image of each connected component $Z'$ of $Z$
is a facet of each of the one or two polytopes corresponding to the
neighboring component(s) of $M \smallsetminus Z$,
and when those are two polytopes, they agree near the facet $\mu (Z')$.
\end{corollary}

{\em Delzant spaces}, also know as {\em symplectic toric manifolds},
are closed symplectic $2n$-dimensional manifolds
equipped with an effective hamiltonian action of an $n$-dimensional
torus and with a corresponding moment map.
Delzant's theorem~\cite{de:hamiltoniens} says that the image
of the moment map (a polytope in $\RR^n$) determines the Delzant space
(up to an equivariant symplectomorphism intertwining the moment maps).
The {\em Delzant conditions} on polytopes are conditions
characterizing exactly those polytopes that occur as moment
polytopes for Delzant spaces.
A polytope in $\RR^n$ is \textit{Delzant} if:
\begin{itemize}
\item
there are $n$ edges meeting at each vertex;
\item
each edge meeting at vertex $p$ is of the form
$p + tu_i$, $t \geq 0$, where $u_i \in \ZZ^n$;
\item
for each vertex, the corresponding $u_1, \ldots, u_n$ can be
chosen to be a $\ZZ$-basis of $\ZZ^n$.
\end{itemize}

Corollary~\ref{cor:agreeing} says that for a toric origami
manifold $(M,\omega,G,\mu)$ the image $\mu(M)$ is the superimposition
of Delzant polytopes with certain compatibility conditions.
Section~\ref{ss:classification} will show how all such
(compatible) superimpositions occur and, in fact, classify
toric origami manifolds.

For a Delzant space, $G$-equivariant symplectic
neighborhoods of connected components of the orbit-type strata
are simple to infer just by looking at the polytope.

\begin{lemma}
\label{lem:delzant}
Let $G=\TT^n$ be an $n$-dimensional torus and $(M_i^{2n},\omega_i,\mu_i)$,
$i=1,2$, two symplectic toric manifolds.
If the moment polytopes $\Delta_i:=\mu_i(M_i)$ agree near facets
$F_1 \subset \mu_1(M_1)$ and $F_2 \subset \mu_2(M_2)$,
then there are $G$-invariant neighborhoods $\cU_i$ of
$B_i = \mu_i^{-1} (F_i)$, $i=1,2$, with a $G$-equivariant
symplectomorphism $\gamma : \cU_1 \to \cU_2$ extending a
symplectomorphism $B_1 \to B_2$ and such that $\gamma^* \mu_2 = \mu_1$.
\end{lemma}

\begin{proof}
Let $\cU$ be an open set containing $F_1=F_2$ such that
$\cU \cap \Delta_1 = \cU \cap \Delta_2$.

Perform symplectic cutting~\cite{le:cuts} on $M_1$ and $M_2$
by slicing $\Delta_i$ along a hyperplane
parallel to $F_i$ such that:
\begin{itemize} 
\item
the moment polytope $\widetilde{\Delta}_i$ containing $F_i$
is in the open set $\cU$;
\item
the hyperplane is close enough to $F_i$ to
guarantee that $\widetilde{\Delta}_i$ satisfies the 
Delzant conditions.
(For generic rational
hyperplanes, the third of the Delzant conditions fails inasmuch as
we only get a $\QQ^n$-basis, thus we need to consider orbifolds.)
\end{itemize}
Then $\widetilde{\Delta}_1=\widetilde{\Delta}_2$.
By Delzant's theorem, the corresponding cut spaces
$\widetilde{M}_1$ and $\widetilde{M}_2$
are $G$-equivariantly symplectomorphic, the symplectomorphism
pulling back one moment map to the other.
Restricting the previous symplectomorphism gives us a 
$G$-equivariant symplectomorphism between $G$-equivariant 
neighborhoods $\cU_i$ of $B_i$ in $M_i$
pulling back one moment map to the other.
\end{proof}

\vspace*{1ex}

\begin{example}
\label{ex:hirzebruch}
The polytopes in Figure~\ref{fig:agreeing} represent four
different symplectic toric 4-manifolds:
twice the topologically nontrivial $S^2$-bundle over $S^2$
(these are Hirzebruch surfaces),
an $S^2 \times S^2$ blown-up at one point
and an $S^2 \times S^2$ blown-up at two points.
If any two of these polytopes are translated so that their left vertical
edges exactly superimpose, we get examples to which Lemma~\ref{lem:delzant}
applies, the relevant facets being the vertical facets on the left.

\begin{figure}[ht]
\begin{center}
\includegraphics[scale=.6]{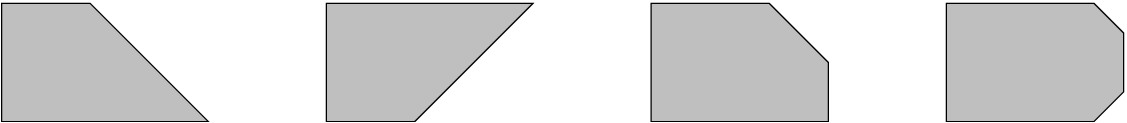}
\caption{Polytopes agreeing near the left vertical edges}
\label{fig:agreeing}
\end{center}
\end{figure}

\end{example}

\begin{example}
\label{flag}
Let $(M_1,\omega_1,\TT^2,\mu_1)$ and $(M_2,\omega_2,\TT^2,\mu_2)$
be the first two symplectic toric manifolds from Example~\ref{ex:hirzebruch}
(Hirzebruch surfaces).
Let $(B,\omega_B,\TT^2,\mu_B)$ be a symplectic $S^2$
with a hamiltonian (noneffective) $\TT^2$-action and
hamiltonian embeddings $j_i$ into $(M_i,\omega_i,\TT^2,\mu_i)$ 
as preimages of the vertical facets.
By Lemma~\ref{lem:delzant}, there exists a $\TT^2$-equivariant
symplectomorphism $\gamma : \cU_1 \to \cU_2$ between invariant
tubular neighborhoods $\cU_i$ of $j_i(B)$ extending a
symplectomorphism $j_1(B) \to j_2(B)$ such that $\gamma^* \mu_2 = \mu_1$.
The corresponding radial blow-up has the origami polytope
in Figure~\ref{fig:flag}.

\begin{figure}[ht]
\begin{center}
\includegraphics[scale=.17]{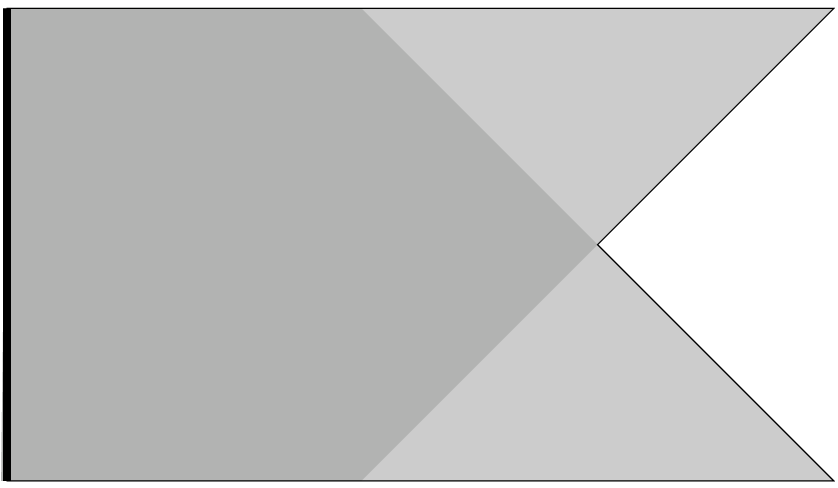}
\caption{Origami polytope for a radial blow-up of two Hirzebruch surfaces}
\label{fig:flag}
\end{center}
\end{figure}

Different shades of gray distinguish regions where
each point represents one orbit (lighter) or two orbits (darker),
which results from the superimposition of two Hirzebruch polytopes.

This example may be considerably generalized;
see Section~\ref{ss:classification}.
\end{example}

\begin{example}\label{exotic}
Dropping the origami hypothesis gives us much less rigid moment map images.
For instance, take any symplectic toric manifold $(M',\omega',G,\mu')$,
e.g.\ $S^2 \times S^2$, and use a regular closed
curve inside the moment image to scoop out a $G$-invariant
open subset corresponding to the region inside the curve.
Let $f$ be a defining function for the curve such that $f$ is
positive on the exterior.
Consider the manifold
\[
   M = \left\{(p,x)\in M' \times \RR \, \mid \, x^2=f(p)\right\} \ .
\]
This is naturally a {\em toric} folded symplectic manifold.
However, the null foliation on $Z$ is not fibrating: at 
points where the slope of the curve is irrational,
the corresponding leaf is not compact.

For instance, when $M' = S^2 \times S^2$ and we
take some closed curve, the moment map image is as on the left
of Figure~\ref{fig:exotic}.
If instead we discard the region corresponding to the outside
of the curve (by choosing a function $f$ positive
on the interior of the curve), the moment map image is as on the right.

\begin{figure}[ht]
\begin{center}
\qquad
\includegraphics[scale=.23]{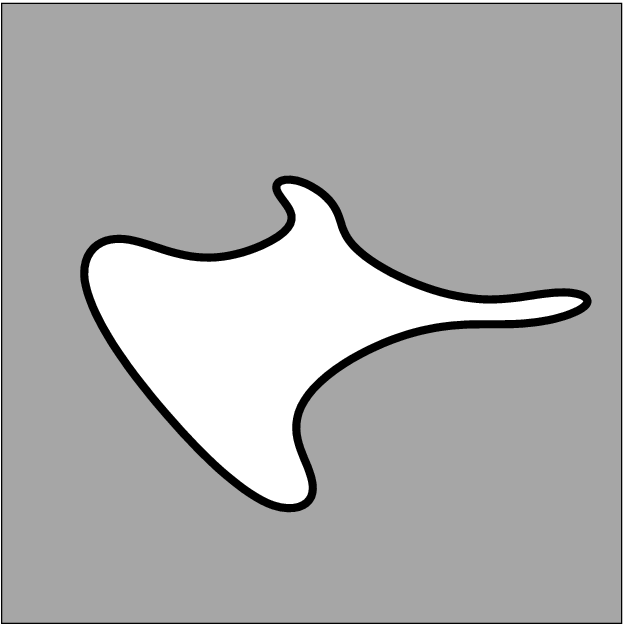}
\qquad \qquad \qquad
\includegraphics[scale=.23]{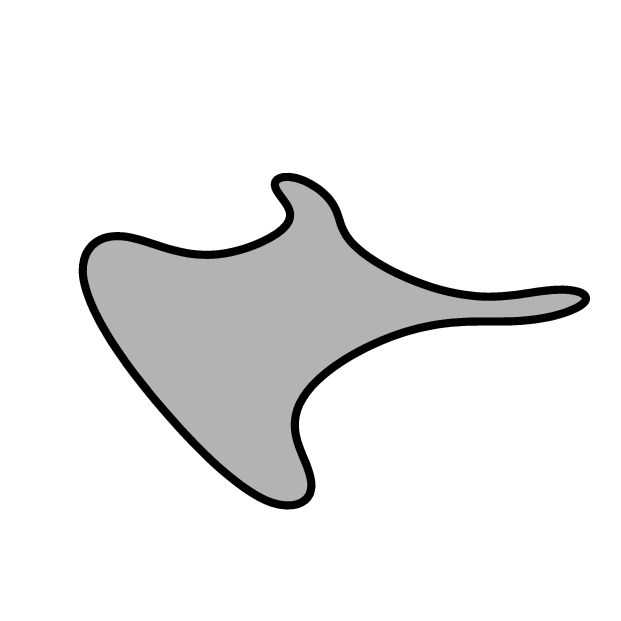}
\caption{Folded (non-origami) moment map images}
\label{fig:exotic}
\end{center}
\end{figure}

\end{example}


\subsection{Classification of toric origami manifolds}
\label{ss:classification}

Let $(M,\omega,G,\mu)$ be a toric origami manifold.
By Theorem~\ref{thm:convexity}, the image
$\mu (M)$ is the superimposition of the Delzant polytopes
corresponding to the connected components of its symplectic cut space.
Moreover, $\mu$ maps the folding hypersurface to certain facets
possibly shared by two polytopes which agree near those facets.

Conversely, we will see that, given a {\em template}
of an allowable superimposition of Delzant polytopes,
we can construct a toric origami manifold
whose moment image is that superimposition.
Moreover, such templates classify toric origami manifolds.

\begin{definition}\label{def:template}
An $n$-dimensional \textit{origami template} is a pair $(\cP,\cF)$,
where $\cP$ is a (nonempty) finite collection
of $n$-dimensional Delzant polytopes and
$\cF$ is a collection of facets and pairs of facets of polytopes in $\cP$
satisfying the following properties:
\begin{itemize}
\item[(a)]
for each pair $\left\{F_1,F_2\right\} \in \cF$, the corresponding
polytopes in $\cP$ agree near those facets;
\item[(b)]
if a facet $F$ occurs in $\cF$, either by itself or as a member
of a pair, then neither $F$ nor any of its neighboring facets
occur elsewhere in $\cF$;
\item[(c)]
the topological space constructed from the disjoint union
$\sqcup\Delta_i$, $\Delta_i\in\cP$, by identifying facet pairs
in $\cF$ is connected.
\end{itemize}
\end{definition}

\begin{theorem}\label{thm:classification}
Toric origami manifolds are classified by origami templates
up to equivariant symplectomorphism preserving the moment maps.
More specifically, at the level of symplectomorphism classes
(on the left hand side), there is a one-to-one correpondence
\[
\begin{array}{rll}
   \left\{\text{$2n$-diml toric origami manifolds}\right\}
   & \longrightarrow
   & \left\{\text{$n$-diml origami templates}\right\}\\
   (M^{2n},\omega,\TT^n,\mu)&\longmapsto&\mu(M).
\end{array}
\]
\end{theorem}

\begin{proof}
To build a toric origami manifold from a template $(\cP,\cF)$,
take the Delzant spaces corresponding to the
Delzant polytopes in $\cP$ and radially blow up the inverse
images of the facets occuring in sets in $\cF$:
for pairs $\left\{F_1,F_2\right\} \in \cF$
the model involution $\gamma$ uses the symplectomorphism from 
Lemma~\ref{lem:delzant}; for single faces $F \in \cF$,
the map $\gamma$ must be the identity.
The uniqueness part follows from an equivariant version of
Corollary~\ref{coroll:blowingup_unfold}.
\end{proof}

\begin{remark}
There is also a one-to-one correspondence between {\em oriented} origami 
toric manifolds (up to equivariant symplectomorphism) and
{\em oriented} origami templates.
We say that an origami template is \textit{oriented}
if the polytopes in $\cP$ come with an orientation and $\cF$ consists solely of pairs of facets which belong to polytopes with opposite orientations; see Introduction.
Indeed, for $\left\{F_1,F_2\right\}\in\cF$, with 
$F_1 \in \Delta_1$ and $F_2\in\Delta_2$, the opposite orientations on the polytopes $\Delta_1$ and $\Delta_2$ induce opposite orientations on the corresponding
components of $M \smallsetminus Z$ which piece together to a global
orientation of $M$, and vice-versa.
\end{remark}

\begin{example}
Unlike ordinary toric manifolds, toric origami manifolds may come 
from non-simply connected templates. Let $M$ be the manifold 
$S^2\times S^2$ blown up at two points, with one $S^2$ factor 
having three times the area of the other: the associated 
polytope $\Delta$ is a rectangle with two corners removed. We can 
construct an origami template $(\cP,\cF)$ where $\cP$ consists 
of four copies of $\Delta$ arranged in a square and $\cF$ is 
four pairs of edges coming from the blowups.
The result is shown in Figure~\ref{fig:hole_polytope}.
Note that the associated origami manifold is also not simply connected.

\vspace*{5ex}

\begin{figure}[ht]
\begin{center}
\includegraphics[scale=.25]{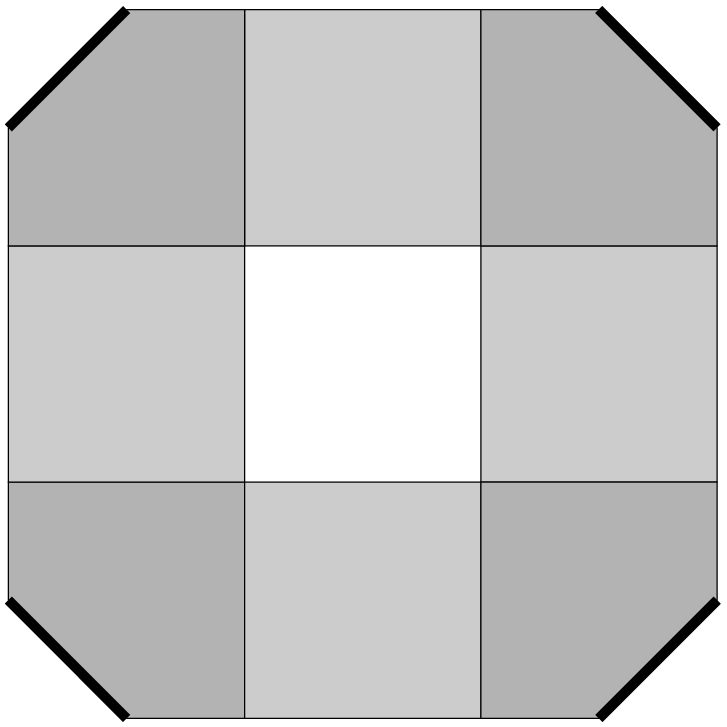}
\caption{Non-simply-connected toric origami template}
\label{fig:hole_polytope}
\end{center}
\end{figure}

\end{example}

\begin{example}
We can form higher-dimensional analogues of the previous example 
which fail to be $k$-connected for $k\geq 2$. In the case $k=2$, 
for instance, let $\Delta'$ be the polytope associated to $M\times S^2$, 
and construct an origami template $(\cP',\cF')$ just as before: 
this gives the three-dimensional figures on the left and right
of Figure~\ref{fig:cubeincube}.
We now superimpose these two solids along the dark shaded 
facets (the bottom facets of the top copies of $\Delta'$), giving us 
a ninth pair of facets and the desired non-$2$-connected template.

\vspace*{5ex}

\begin{figure}[ht]
\begin{center}
\includegraphics[scale=.33]{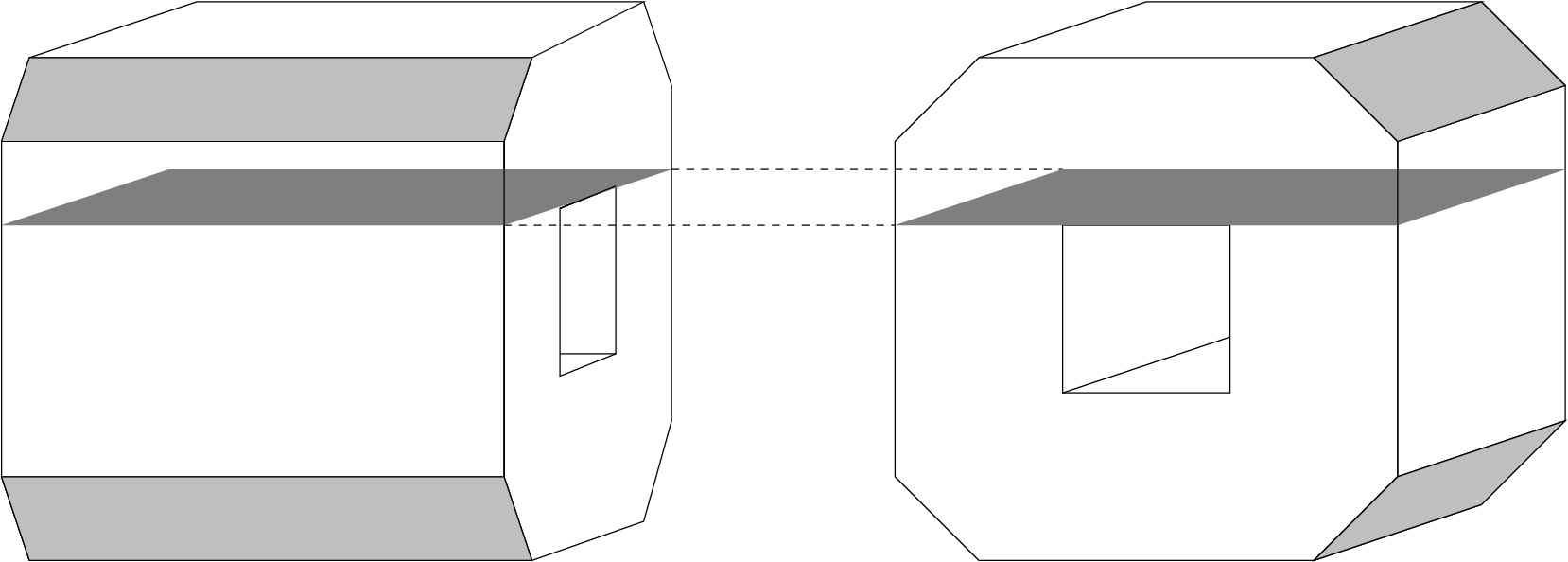}
\caption{Non-2-connected toric origami template}
\label{fig:cubeincube}
\end{center}
\end{figure}

Note that even though the moment map image has interesting $\pi_2$, the template (thought of as the polytopes glued along facets) has trivial $H_2$ and $\pi_2$. This is indeed a general feature of origami templates, see Remark~\ref{nonorigami}.
\end{example}

\begin{example}
Although the facets of $\cF$ are necessarily paired if the 
origami manifold is oriented, the converse fails.
As shown in Figure~\ref{fig:nonorientable},
one can form a template of three polytopes, each 
corresponding to an $S^2\times S^2$ blown up at two points, 
and three paired facets. Since each fold flips orientation, 
the resulting topological space is nonorientable.

\begin{figure}[ht]
\begin{center}
\includegraphics[scale=.27]{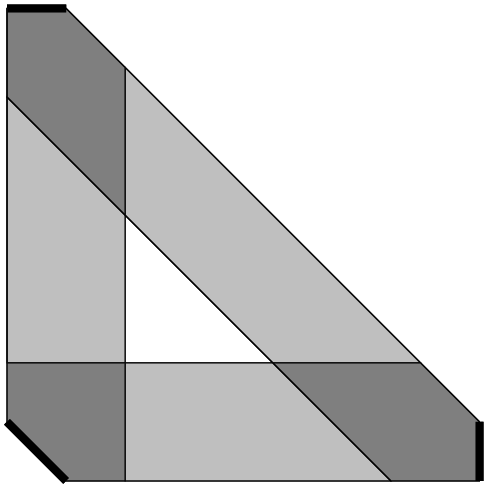}
\caption{Non-orientable toric origami template with co-orientable folds}
\label{fig:nonorientable}
\end{center}
\end{figure}

\end{example}

\begin{example}
Recall that the polytope associated to $\CC \PP^2$ is a triangle 
(shown on the right of Figure~\ref{fig:three_triangles}).
The sphere $S^4$ (shown left) is the orientable 
toric origami manifold whose template is two copies of this 
triangle glued along one edge. Similarly, $\RR \PP^4$ (shown center) 
is the nonorientable manifold whose template is a single copy of 
the triangle with a single folded edge. This exhibits $S^4$ as a 
double cover of $\RR \PP^4$ at the level of templates.

\vspace*{4ex}

\begin{figure}[ht]
\begin{center}
   \psfrag{fig1}{$S^4$}
   \psfrag{fig2}{$\RR\PP^4$}
   \psfrag{fig3}{$\CC\PP^2$}
\includegraphics[scale=.35]{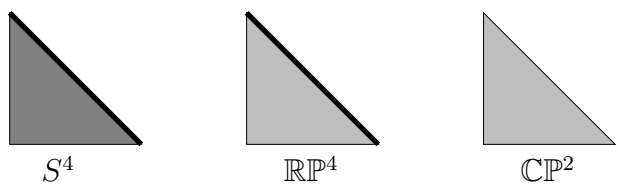}
\caption{Three origami templates with the same origami polytope}
\label{fig:three_triangles}
\end{center}
\end{figure}

\end{example}

\begin{example}
We can classify all two-dimensional toric origami manifolds by 
classifying one-dimensional templates
(Figure~\ref{fig:zigzag_spheres}--~\ref{fig:zigzag_tori}).
These are disjoint unions of $n$ segments (1-dimensional Delzant polytopes)
connected at vertices  (the facets of those polytopes) with zero angle:
internal vertices and endpoints marked with bullets represent folds.
Each segment (resp.\ marking) gives a component 
of $M\smallsetminus Z$ (resp.\ $Z$), while each unmarked endpoint 
corresponds to a fixed point.
There are four families (Instead of 
drawing segments superimposed, we open up angles slightly to show 
the number of components. All pictures ignore segment 
lengths which account for continuous parameters of symplectic area 
in components of $M \smallsetminus Z$.):

\begin{itemize}

\item
Templates with two unmarked endpoints give manifolds diffeomorphic 
to $S^2$: they have two fixed points and $n-1$ components of $Z$
(Figure~\ref{fig:zigzag_spheres}).

\vspace*{3ex}

\begin{figure}[ht]
\begin{center}
\includegraphics[scale=.35]{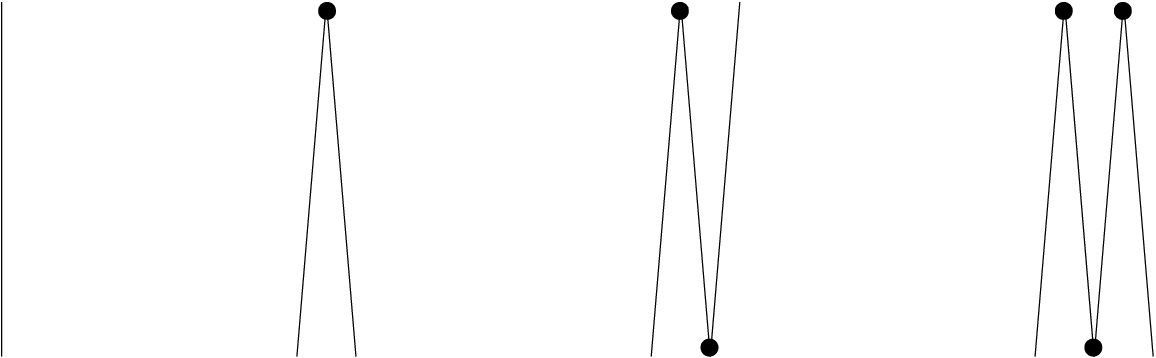}
\caption{Toric origami 2-spheres}
\label{fig:zigzag_spheres}
\end{center}
\end{figure}

\item
Templates with one marked and one unmarked endpoint give manifolds 
diffeomorphic to $\RR \PP^2$: they have one fixed point and $n$ 
components of $Z$ (Figure~\ref{fig:zigzag_rptwos}).

\vspace*{3ex}

\begin{figure}[ht]
\begin{center}
\includegraphics[scale=.35]{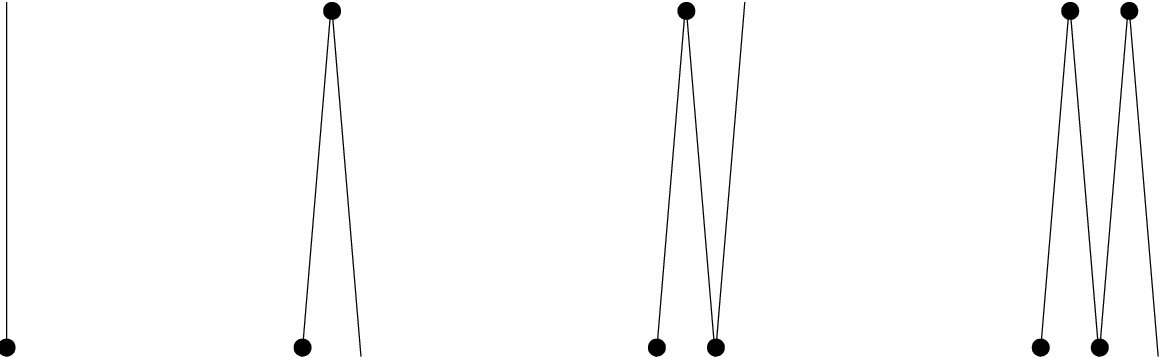}
\caption{Toric origami real projective planes}
\label{fig:zigzag_rptwos}
\end{center}
\end{figure}

\item
Templates with two marked endpoints give manifolds diffeomorphic 
to the Klein bottle: they have no fixed points and $n+1$ components of $Z$
(Figure~\ref{fig:zigzag_klein}).

\vspace*{3ex}

\begin{figure}[ht]
\begin{center}
\includegraphics[scale=.35]{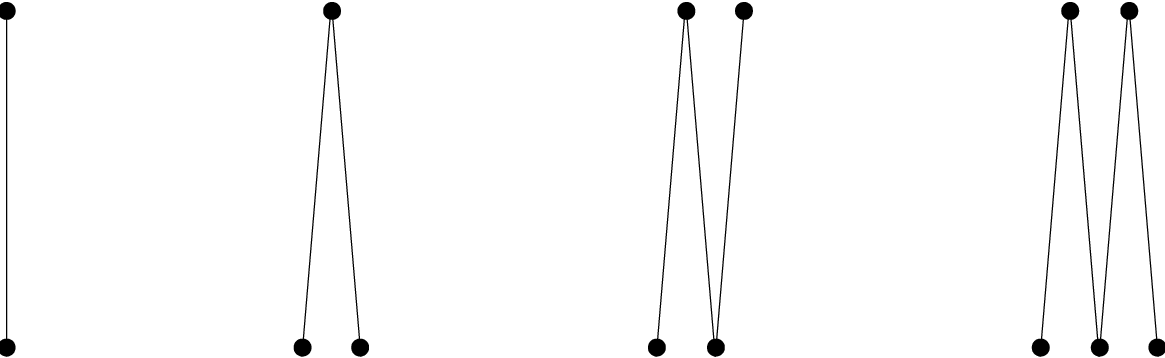}
\caption{Toric origami Klein bottles}
\label{fig:zigzag_klein}
\end{center}
\end{figure}

\item
Templates with no endpoints give manifolds diffeomorphic 
to $\TT^2$: they have no fixed points and
an even number $n$ of components of $Z$ (Figure~\ref{fig:zigzag_tori}).

\vspace*{3ex}

\begin{figure}[ht]
\begin{center}
\includegraphics[scale=.35]{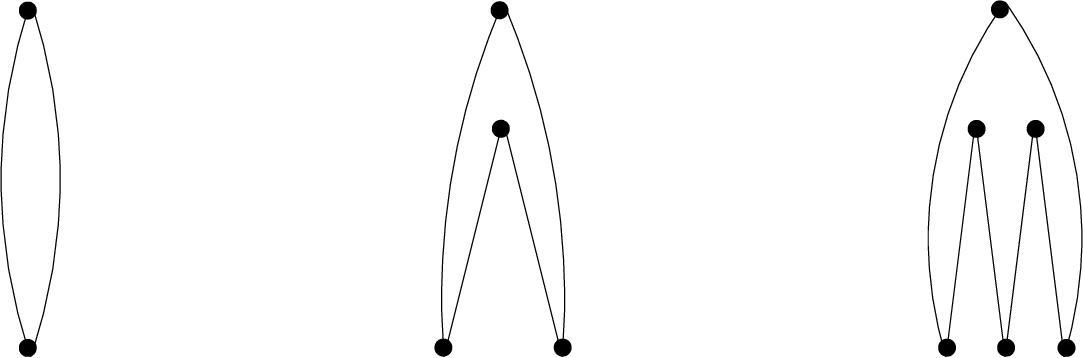}
\caption{Toric origami 2-tori}
\label{fig:zigzag_tori}
\end{center}
\end{figure}
\end{itemize}
\end{example}

\begin{remark}\label{nonorigami}
As illustrated by Example~\ref{exotic}, it is not possible to
classify toric folded symplectic manifolds by combinatorial moment
data as is the case for toric origami manifolds. Such a
classification must be more intricate for the general folded case, and
in~\cite{lee:thesis} C. Lee gives a partial result that sheds some light on
the type of classification that might be possible:

A \emph{toric folded symplectic manifold} $(M,\omega,\TT,\mu)$ is a compact
connected folded symplectic manifold $(M^{2n},\omega)$ endowed with an
effective hamiltonian action of a half-dimensional torus $\TT^n$ and a
corresponding moment map $\mu$. The \emph{orbital moment map} is the
map on the orbit space $M/\TT$ induced by the moment map. Two toric
folded symplectic 4-manifolds $(M,\omega,\TT,\mu)$ and
$(M',\omega',\TT,\mu')$ are symplectomorphic (as folded symplectic manifolds) if $H^2(M/\TT,\ZZ)=0$ and there exists a diffeomorphism between orbit spaces preserving orbital moment maps.

When $(M,\omega,\TT,\mu)$ is a toric origami manifold, $M/\TT$ can be
realized as the topological space obtained by identifying the 
polytopes of its origami template along the common facets (see point (c) in Definition~\ref{def:template}). This space has the same homotopy type as the graph obtained by replacing each polytope by a point and each
``glued'' double facet by an edge between the points corresponding to 
the polytopes that the facet belongs to. Therefore, $H^2(M/\TT,\ZZ)=0$.
The existence of a diffeomorphism between orbit spaces implies that
$(M',\omega',\TT,\mu')$ is an origami manifold as well, and that its
origami template is the same as that of $(M,\omega,\TT,\mu)$, which
makes the manifolds symplectomorphic by Theorem~\ref{thm:classification}.
\end{remark}


\section{Cobordism}
\label{cobordism}

We will now prove the following conjecture of Yael Karshon's
stating that an oriented origami manifold is \textit{symplectically}
cobordant to its symplectic cut space.
By \textit{symplectic} cobordism we mean,
following~\cite{gu-gi-ka:cobordisms},
a cobordism manifold endowed with a closed 2-form which
restricts to the origami and symplectic forms on its boundary.

\begin{theorem}
\label{thm:karshon}
Let $(M,\omega)$ be an oriented origami manifold
and let $(M_0^\pm,\omega_0^\pm)$ be its symplectic cut pieces.

Then there is a manifold $W$ equipped with a closed
2-form $\Omega$ such that the boundary of $W$
equipped with the restriction of $\Omega$ is symplectomorphic to
\[
   (M,\omega) \sqcup (M_0^+,\omega_0^+) \sqcup (M_0^-,\omega_0^-) \ .
\]
Moreover, in the presence of a (hamiltonian) compact group action,
this cobordism can be made equivariant (or hamiltonian).
\end{theorem}

\begin{proof}
Choose a $S^1$-action making the null fibration into a principal
fibration, $S^1\hookrightarrow Z\stackrel{\pi}{\to} B$.
Let $\LL\stackrel{\pi_\LL}{\to}B$ be the associated hermitian line bundle
$Z\times_{S^1}\CC$ for the standard multiplication action of $S^1$ on $\CC$.
Let $r:\LL\to\RR$ given by
$r(\ell)=\sqrt{\left\langle \ell,\ell\right\rangle_{\pi_\LL(\ell)}}$
be the hermitian length and let
$i_Z:Z\hookrightarrow\LL$, $i_Z (x) = [x,1]$.

For $\varepsilon$ small enough, let $\varphi:Z\times(-\varepsilon,\varepsilon)\to\cU$ be a Moser model for $M$ (see Definition \ref{def:mosermodel}). Take a small enough $\delta$ and choose a non-decreasing smooth function $g:\RR^+_0\to\RR$ such that $g(s)=s$ for $s<\delta^2$ and $g(s)=1$ for $s>4\delta^2$. 

The set 
$$\left\{(r,t)\in\RR_0^+\times(-\varepsilon,\varepsilon)\mid g(t^2)-\frac{\delta^2}{4}\leq r^2\leq 1\right\}$$
is depicted in Figure~\ref{fig:cobordism_aux}, where $R:=\sqrt{1-\frac{\delta^2}{4}}$. 

\begin{figure}[ht]
\begin{center}
  \psfrag{r}{$r$}
  \psfrag{r=1}{$r=1$}
  \psfrag{r=R}{$r=R$}
  \psfrag{t}{$t$}
  \psfrag{d}{$\delta$}
  \psfrag{-2d}{$-2\delta$}
  \psfrag{2d}{$2\delta$}
  \psfrag{-d}{$-\delta$}
  \psfrag{d/2}{$\delta/2$}
  \psfrag{-d/2}{$-\delta/2$}
  \psfrag{e}{$\varepsilon$}
  \psfrag{-e}{$-\varepsilon$}\includegraphics[scale=.6]{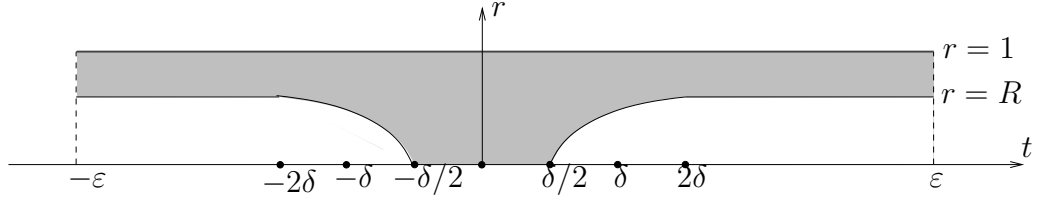}
\caption{The shaded area is the set $\left\{(r,t)\in\RR_0^+\times(-\varepsilon,\varepsilon)\mid g(t^2)-\frac{\delta^2}{4}\leq r^2\leq 1\right\}$}
\label{fig:cobordism_aux}
\end{center}
\end{figure}

Then the set
\[
   W_\varepsilon :=
   \left\{(\ell,t)\in\LL\times(-\varepsilon,\varepsilon)
   \mid g(t^2)-\frac{\delta^2}{4}\leq r(\ell)^2\leq 1\right\}
\]
is the manifold with boundary sketched in
Figure~\ref{figure:cobordism} with $B$ represented by a point. 

\begin{figure}[ht]
\begin{center}
  \psfrag{M}{$C$}
  \psfrag{M-}{$C^-$}
  \psfrag{M+}{$C^+$}
  \psfrag{r=1}{$r=1$}
  \psfrag{r=sqrt(1-d2/4)}{$r=R$}
  \psfrag{t}{$t$}
  \psfrag{d}{$\delta$}
  \psfrag{-2d}{$-2\delta$}
  \psfrag{2d}{$2\delta$}
  \psfrag{-d}{$-\delta$}
  \psfrag{d/2}{$\delta/2$}
  \psfrag{-d/2}{$-\delta/2$}
  \psfrag{t=e}{$t=\varepsilon$}
  \psfrag{t=-e}{$t=-\varepsilon$}
\includegraphics[scale=.6]{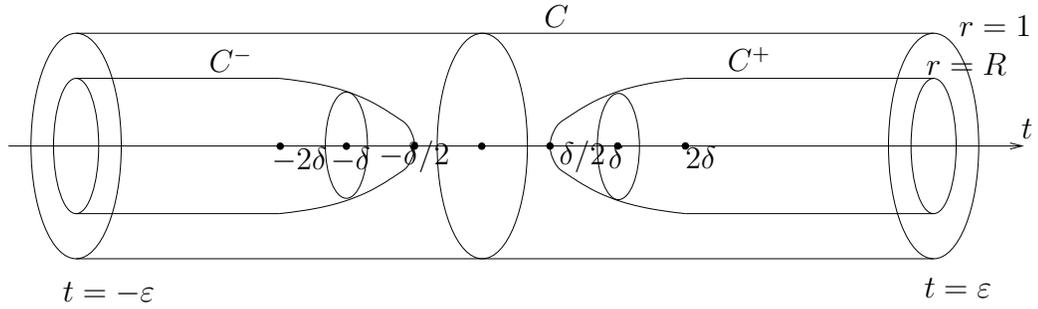}
\caption{The key portion $W_\varepsilon$ of the cobordism manifold $W$}
\label{figure:cobordism}
\end{center}
\end{figure}

The boundary of $W_\varepsilon$ is made up of the following three pieces:
\begin{eqnarray*}
C&:=&\left\{(\ell,t)\in\LL\times(-\varepsilon,\varepsilon)\mid r(\ell)=1\right\}      \\
C^+
& := &
\left\{(\ell,t)\in\LL\times[\textstyle{\frac{\delta}{2}},\varepsilon)
\mid g(t^2)=r(\ell)^2+\frac{\delta^2}{4}\right\}    \\
C^-
& := &
\left\{(\ell,t)\in\LL\times(-\varepsilon,-\textstyle{\frac{\delta}{2}}]
\mid g(t^2)=r(\ell)^2+\frac{\delta^2}{4}\right\} 
\end{eqnarray*}

The set $C$ is the image of $\cU$ under the diffeomorphism 
\begin{equation}\label{eq:diffeo for C}
\cU\stackrel{\varphi^{-1}}{\longrightarrow}Z\times(-\varepsilon,\varepsilon)\stackrel{i_Z\times\text{id}}{\longrightarrow}\LL\times(-\varepsilon,\varepsilon).
\end{equation}

By the tubular neighborhood theorem,
the set $C^+$ is diffeomorphic to a neighborhood of $B$ in the symplectic cut space $M_0^+$ (and similarly for $C^-$ and $M_0^-$); indeed the normal bundle of $B$ in $M_0^+$ is also $\LL$.

We can now extend the cobordism $W_\varepsilon$ between $C$ and $C^+\sqcup C^-$ to a global cobordism between $M$ and $M_0^+\sqcup M_0^-$ (modulo diffeomorphisms). Let $M_{2\delta}:=\varphi^{-1}(Z\times(-2\delta,2\delta))$ be a narrower tubular neighborhood of $Z$ in $M$. We form the global cobordism $W$ by gluing $W_\varepsilon$ to  $\left(M\smallsetminus M_{2\delta}\right)\times\left[R,1\right]$ using the restriction to $\cU\smallsetminus M_{2\delta}$ of the diffeomorphism (\ref{eq:diffeo for C}) and using the identity map on $\left[R,1\right]$. Note that any subset of $\cU$ not containing $Z$, for example $\cU\smallsetminus M_{2\delta}$, can be viewed also as a subset of $M_0^+\sqcup M_0^-$ via $j^+\sqcup j^-$ (for $j^+$ and $j^-$ see proof of Proposition \ref{pro:cutting}).

We will next exhibit a closed 2-form on $W$ restricting to
the given origami and symplectic forms on the boundary.

Let $\alpha$ be the $S^1$-connection form on $Z$ from the Moser model.
Denote by $B_0$ the zero-section in $\LL$.
Since $\LL \smallsetminus B_0 = Z\times_{S^1}\CC^* \approx Z\times\RR^+$,
we can extend $\alpha$ to $\LL \smallsetminus B_0$ by pull-back
via the projection $Z\times\RR^+\to Z$.
Although $\alpha$ is not defined over $B_0$,
the product $r^2\alpha$ is a smooth 1-form on $\LL$:
On open sets where the fibration $Z\stackrel{\pi}{\rightarrow}B$ is trivial,
we have
$\alpha=d\theta+\pi^*\xi$ for some $\xi\in\Omega^1(B)$, and $r^2\alpha=r^2d\theta+r^2\pi^*\xi$ is well defined on the whole $\LL$ because $r^2 d\theta$ is no longer singular.

Now let $h(r,t)$ be a smooth function defined on the shaded region
in Figure~\ref{fig:cobordism_aux} which is even in $t$
(i.e., $h(r,t)=h(r,-t)$) and in particular equal to $t^2$ on the
region $|t|>2\delta$ and on the line $r=1$, and whose restriction
$\tilde{h}$ to the curve $g(t^2)=r^2+\frac{\delta^2}{4}$ is strictly
increasing as a function of $t$ and equal to
$r^2$ for $|t|<\delta$.
An exercise in 2-D interpolation shows such a function exists;
note that, by definition of $g$, we have that
$r^2=t^2-\frac{\delta^2}{4}$ on the latter curve when $|t|<\delta$.

We define the following closed 2-form on $W_\varepsilon$:
$$\Omega=\pi^*\omega_B+d(h\alpha).$$
Because the restriction $\tilde{h}$ is unique up to pre-composition
with a diffeomorphism, the restrictions of $\Omega$ to $C^+$ and $C^-$
are unique up to symplectomorphism.
Furthermore, by the Weinstein tubular neighborhood
theorem~\cite{we:symplecticmanifolds}, $\Omega|_{C^+}$ and $\Omega|_{C^-}$
are, up to symplectomorphism, the symplectic forms
$\omega_0^+$ and $\omega_0^-$ on corresponding open subsets
of $M_0^+$ and $M_0^-$.
On the other hand, restricting to $C$, we have
$\Omega=\pi^*\omega_B+d(t^2\alpha)$ which under the map in
(\ref{eq:diffeo for C}) is the origami form on $\cU$.

In particular, on the region $|t|>2\delta$ in Figure~\ref{figure:cobordism}, the folded symplectic form on $C$, the symplectic forms on $C^+$ and $C^-$, and the cobording form $\Omega$ on $W_\varepsilon$ are all identical and equal to $\pi^*\omega_B+d t^2\alpha$. Moreover, the ``up to symplectomorphism'' part of these statements involves symplectomorphisms which are the identity on the region $|t|>2\delta$, and hence $\Omega$ can be extended to a cobording 2-form on all of $W$ by letting it be the constant (constant on $r$) cobording 2-form on $|t|>\varepsilon$.

Suppose now that $M$ is equipped with an action of a compact Lie group
$G$ which preserves $\omega$.
One then gets a $G$-action on $B$, $Z$ and $\LL$, and
by averaging we can choose the $\alpha$ in the Moser model to be $G$-invariant.
Thus all the data involved in the definition of the cobording 2-form
$\Omega$ above are $G$-invariant and hence $\Omega$ itself is $G$-invariant.
Moreover, if $\omega$ is $G$-hamiltonian, the form $\omega_B$ is as well,
and since $\alpha$ is $G$-invariant the 2-form $d(h\alpha)$
is $G$-hamiltonian with moment map
\[
   v\in\fg \longmapsto \phi^v=\imath_{v^\#}(h\alpha)
\]
where $v^\#$ is the vector field on $\LL\times(-\varepsilon,\varepsilon)$
associated with the action of $G$ on this space.
Thus $W_\varepsilon$ with the form $\Omega$ is a hamiltonian cobordism, and it extends to a hamiltonian cobordism $W$.
\end{proof}

\begin{remark}
If $M$ and its cut pieces are pre-quantizable one has an isomorphism of virtual vector spaces (and in the presence of group actions, virtual representation spaces)
$$\mathcal{Q}(M)=\mathcal{Q}(M_0^+)-\mathcal{Q}(M_0^-)$$
where $\mathcal{Q}$ is the spin-$\CC$ quantization functor.
For a proof of this see~\cite{ca-gu-wo:unfolding}.
An alternative proof of this result is based on the
``quantization commutes with cobordism'' theorem
of~\cite{gu-gi-ka:cobordisms}.
\end{remark}

\begin{remark}
By the ``cobordism commutes with reduction'' theorem
of~\cite{gu-gi-ka:cobordisms}, the symplectic reduction of $M$
at a regular level (for an hamiltonian abelian Lie group action)
is cobordant to the symplectic reduction of $M_0^+\sqcup M_0^-$ at that level.
\end{remark}

\begin{remark}
In the nonorientable case, we would obtain an orbifold cobordism.
However, this is not interesting, since any manifold $M$ bounds
an orbifold, $\left( M \times [-1,1] \right) / \ZZ_2$.
\end{remark}


\section{Cohomology of Toric Origami}
\label{sec:cohomology}

In this section we assume connectedness of the folding hypersurface $Z$. This assumption is essential for the argument below: the more general case of a nonconnected folding hypersurface remains open.

Let $(M,\omega,\TT,\mu)$ be a $2n$-dimensional oriented
toric origami manifold with null fibration
$Z \stackrel{\pi}{\twoheadrightarrow} B$
and connected folding hypersurface $Z$.
Let $S^1\subset\TT$ be the circle group generating the null fibration and $f:M\rightarrow\RR$ a corresponding moment map with $f=0$ on $Z$ and $f>0$ on $M\smallsetminus Z$. Note that, on a tubular neighborhood of $Z$ given by a Moser diffeomorphism $\varphi:Z\times(-\e,\e)\rightarrow\cU$, the origami form is $\varphi^*\omega=p^*i^*\omega+d(t^2 p^*\alpha)$ and hence the moment map is $\varphi^*f(x,t)=\frac{t^2}{2}$. 

Near the folding hypersurface $Z$, the function $f$ is essentially $\frac{t^2}{2}$, and away from it, $f$ restricts to a moment map on an honest symplectic manifold $M\smallsetminus Z$, and is thus Morse-Bott. Furthermore, $Z$ is a nondegenerate critical manifold of codimension one.

Define $g:M\rightarrow\RR$ as 
$$ g = \left\{ \begin{array}{ll}
         \sqrt{f} & \mbox{on $M^+$}\\
         0 & \mbox{on $Z$}\\
        -\sqrt{f} & \mbox{on $M^-$}\end{array} \right. $$

We claim that $g$ is Morse-Bott and its critical manifolds are those of $f$, excluding $Z$. But this follows easily from the fact that, on $\cU$, we have $\varphi^*g=\frac{t}{\sqrt{2}}$, whose derivative never vanishes, while on $M\smallsetminus Z$, $dg$ vanishes if and only if $df$ vanishes:
$$dg=\left\{ \begin{array}{ll}
         df/(2\sqrt{f}) & \mbox{on $M^+$}\\
        -df/(2\sqrt{f}) & \mbox{on $M^-$}\end{array} \right.$$ 

Morse(-Bott) theory then gives us the cohomology groups $H_{\TT}^k(M)$
in terms of the cohomology groups of the critical manifolds of $g$,
$X\subset M\smallsetminus Z$:
\[
   H_{\TT}^k(M)=\left\{ \begin{array}{ll}
         0 & \mbox{if $k$ odd}\\
        \sum_{X} H_{\TT}^{k-r_X}(X) & \mbox{if $k$ even}\end{array} \right.
\]
where $r_X=\text{Ind}(X,g)$ is the index of the critical manifold $X$
with respect to the function $g$, and is $\text{Ind}(X,f)$ if $X\subset M^+$,
and $2 d-\text{Ind}(X,f)$ if $X\subset M^-$.

Similarly to the symplectic toric manifold case,
the cohomology groups are easily read from the template.

Let $(\cP,\cF)$ be the template of a $2n$-dimensional oriented
toric origami manifold $(M,\omega,\TT,\mu)$
with connected folding hypersurface $Z$.
Since the manifold is oriented, the set $\cP$ is a finite collection
of oriented $n$-dimensional Delzant polytopes
and the set $\cF$ a collection of pairs
of facets of these polytopes (see the Introduction).
We say that a polytope is \textit{positively oriented}
if its orientation matches that of ${\mathfrak g}^*$
and \textit{negatively oriented} otherwise.
Let $S^1\subset\TT$ be the circle group generating the null fibration
and let $\HH\subset\TT$ be a complementary $(n-1)$-dimensional subtorus.
Let ${\mathfrak s}^* \oplus {\mathfrak h}^*$ be the induced decomposition
of the dual of the Lie algebra of $\TT$, and let
$\mathrm{refl} (a \oplus b) = -a \oplus b$ be the
corresponding reflection along ${\mathfrak s}^*$.
Let $\cP^{u}$ be the collection of all positively oriented
polytopes in $\cP$ and of the images by $\mathrm{refl}$
of all negatively oriented polytopes in $\cP$.
The set $\cP^{u}$ can be thought of in terms of unfolding
of the moment polytope.
Let $\cF^{u}$ be the set of pairs of facets of polytopes in $\cP^{u}$
corresponding to the pairs in $\cF$.
We will call $(\cP^{u},\cF^{u})$ the
\textit{unfolded template} of $(M,\omega,\TT,\mu)$.

\begin{corollary}
\label{coroll:cohomology}
Let $(\cP^{u},\cF^{u})$ be the unfolded template of
a $2n$-dimensional oriented toric origami manifold $(M,\omega,\TT,\mu)$
with connected folding hypersurface $Z$, as defined above.
Let $X \in \fg$ generate an irrational flow.

For $k$ even, the degree-$k$ cohomology group of $M$ has dimension equal
to the number of vertices $v$ of polytopes in $\cP^{u}$ such that:
\begin{itemize}
\item[(i)]
at $v$ there are exactly $\frac{k}{2}$ primitive inward-pointing
edge vectors which point up relative to the projection along $X$, and
\item[(ii)]
$v$ does not belong to any facet in $\cF^{u}$.
\end{itemize}
All odd-degree cohomology groups of $M$ are zero.
\end{corollary}

\begin{example}
\label{ex:cohomology_S4_1}
For the toric 4-sphere $(S^4, \omega_0, \TT^2, \mu)$
from Example~\ref{ex:toric_sphere}, the set $\cP^{u}$ contains
two triangles, one being a mirror image of the other, as in
Figure~\ref{fig:unfolded_triangle}, where the dashed hypotenuses
form the unique pair of facets in the corresponding $\cF^{u}$.

\begin{figure}[ht]
\begin{center}
   \psfrag{0}{$0$}
   \psfrag{2}{$2$}
   \psfrag{X}{$X$}
\includegraphics[scale=.6]{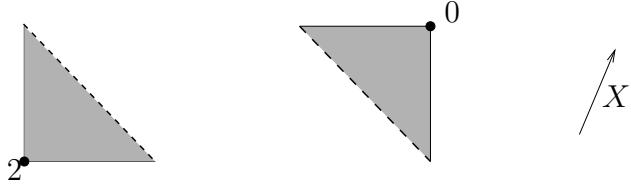}
\caption{The {\em unfolded} set $\cP^{u}$ for a toric 4-sphere}
\label{fig:unfolded_triangle}
\end{center}
\end{figure}

For the chosen direction $X$, the numbers next to the
relevant vertices count the edge vectors which point up relative to $X$.
Indeed, $\dim H^4(S^4) = \dim H^0(S^4) = 1$,
all other groups being trivial.
\end{example}

\begin{example}
\label{ex:cohomology_flag_1}
For the toric 4-manifold $(M, \omega, \TT^2, \mu)$ from Example~\ref{flag},
the set $\cP^{u}$ contains
two trapezoids, as in Figure~\ref{fig:unfolded_flag},
where the dashed vertical sides 
form the unique pair of facets in its $\cF^{u}$.

\begin{figure}[ht]
\begin{center}
   \psfrag{0}{$0$}
   \psfrag{1}{$1$}
   \psfrag{2}{$2$}
   \psfrag{X}{$X$}
\includegraphics[scale=.7]{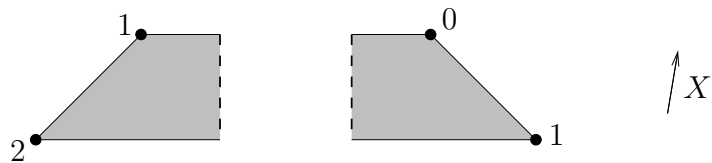}
\caption{The {\em unfolded} set $\cP^{u}$ for a toric manifold with a flag-like moment polytope}
\label{fig:unfolded_flag}
\end{center}
\end{figure}

For the chosen direction $X$, the numbers next to the
relevant vertices count the edge vectors which point up relative to $X$.
The conclusion is that
\[
   H^k(M;\ZZ)=\left\{ \begin{array}{ll}
         0 & \mbox{if $k$ odd}\\
         \ZZ & \mbox{if $k=0$ or $k=4$}\\
         \ZZ^2 & \mbox{if $k=2$}\end{array} \right.
\]
which happen to coincide with the groups for an ordinary Hirzebruch surface.
\end{example}


\subsection*{Funding}

The first author was partially supported by the
Funda\c{c}\~{a}o para a Ci\^{e}ncia e a Tecnologia (FCT/Portugal).
The third author was partially supported by FCT grant
SFRH/BD/21657/2005.


\subsection*{Acknowledgements}

We would like to express our gratitude to friends, colleagues and referees
who have helped us write and rewrite sections of this paper and/or
have given us valuable suggestions about the contents.
We would particularly like to thank in this regard Yael Karshon,
Allen Knutson, Chris Lee, Sue Tolman, Kartik Venkatram and Alan Weinstein.



\end{document}